\documentclass{amsart}
\usepackage[utf8]{inputenc}
\usepackage{url,hyperref}
\usepackage{amsfonts,latexsym,amssymb,amsmath,amsthm}
\usepackage{biblatex}
\usepackage{bm}
\usepackage{graphicx} 
\usepackage{caption}
\usepackage{enumitem}  
\usepackage{rotating}
\usepackage{tikz}

\usetikzlibrary{calc}
\usetikzlibrary{scopes}
\usetikzlibrary{backgrounds}

\theoremstyle{plain}
\newtheorem{theorem}{Theorem}[section]
\newtheorem{proposition}[theorem]{Proposition}
\newtheorem{corollary}[theorem]{Corollary}
\newtheorem{lemma}[theorem]{Lemma}
\newtheorem*{conjecture*}{Conjecture}

\theoremstyle{definition}
\newtheorem{definition}[theorem]{Definition}
\newtheorem{example}[theorem]{Example}
\newtheorem{remark}[theorem]{Remark}

\numberwithin{equation}{section}

\def\N{\mathbb N}
\def\R{\mathbb R}
\def\WH{\mathbb{WH}}

\newcommand{\lset}[1]{\mathsf{#1}}

\def\des{\textup{des}}

\def\A{\mathcal{A}}
\def\B{\mathcal{B}}

\def\T{\mathcal{T}}

\def\newop#1{\expandafter\def\csname #1\endcsname{\mathop{\rm #1}\nolimits}}

\newop{lex}

\begin{document}

\title{A M\"obius function computing $h$-polynomials of nestohedra}

\begin{abstract}
We introduce a poset of weighted hypergraphs on a fixed vertex set which has a remarkable property: the M\"obius values of the intervals associated to any hypergraph give, up to sign, the coefficients of the $h$-polynomial of its nestohedron. We exploit this relation to provide two new recurrence relations to compute $h$-polynomials of nestohedra.
\end{abstract}

\author{Nicolas Avila Ramirez}

\address[N. Avila Ramirez]{D\'{e}partement de math\'{e}matiques, Universit\'e du Qu\'ebec \`a Montr\'eal, \linebreak  Montr\'eal, Canada.}
\email{avila\_ramirez.nicolas@courrier.uqam.ca}

\author{Sergio A. Carrillo}

\address[S. A. Carrillo]{Departamento de Matemáticas, Universidad Nacional de Colombia, Bogot\'{a}, Colombia.}

\email{sacarrillot@unal.edu.co}

\author[Gonz\'{a}lez D'Le\'{o}n]{Rafael S. Gonz\'alez D'Le\'on$^{*}$}
\address[R. S. Gonz\'{a}lez D'Le\'{o}n]{Department of Mathematics and Statistics, Loyola University of Chicago, Chicago, USA}
\email{rgonzalezdleon@luc.edu}
\urladdr{\url{https://dleon.combinatoria.co/}}
\thanks   {$^{*}$ Partially supported by an AMS-Simons Research Enhancement Grant for Primarily Undergraduate Institution Faculty.}

\subjclass[2020]{Primary: 06A07, 05C65, 52B05 Secondary: 52B12}

\keywords{Hypergraphs, building sets, nestohedra, $h$-polynomials,  M\"obius functions}

\maketitle

\section{Introduction}\label{section:introduction}

Building sets were introduced in the work of De Concini and Procesi on wonderful compactifications of complements of complex hyperplane arrangements \cite{DeConciniProcesi1995}. They capture the notion of connectivity present in many mathematical contexts and have been generalized further from Boolean algebras to general semilattices in the work of Feichtner and Kozlov in \cite{FeichtnerKozlov2004}.    

A \textit{building set}	on ground set $V$ is a collection $\B\subseteq  \mathcal{P}(V)$ of nonempty subsets (\emph{blocks}) of $V$ satisfying that
	\begin{enumerate}[label=(\textbf{B\arabic*})]
		\item \label{B1} if $I, J\in\B$ and $I\cap J \not = \emptyset$, then $I\cup J\in\B$.
		
		\item \label{B2} $\B$ contains the singleton $\{v\}$, for all $v\in V$.
	\end{enumerate}

    The most natural examples of building sets are perhaps the ones associated to connectivity in graphs. For a graph $G=(V,E)$ on vertex set $V$ and set of edges $E=E(G)$, the family $\B(G)$ formed by subsets $I \subseteq V$ such that the induced subgraph of $G$ on vertex set $I$ is connected, is a building set. Building sets of the form $\B(G)$ are known as \emph{graphical}. 
    
    In a greater generality, recall that a \emph{(multi) hypergraph} $H$ on $V=:V(H)$, is a collection of (possibly repeated) nonempty subsets of $V$, called \emph{hyperedges}, where any two hyperedges on the same underlying vertex set are considered to be distinct---to have different identities. For a hypergraph $H$ on $V$, the collection $\B(H)$ of subsets of $V$ that are connected by elements of $H$ (see Section \ref{section:poset_of_weighted_hypergraphs} for the definition) also forms a building set. Every building set is of the form $\B=\B(H)$ for some hypergraph $H$, but this map is not a bijection since there are infinitely many $H$ satisfing this property. In fact, we can consider building sets $\B$ as hypergraphs themselves, and in such case---after an abuse of notation---$\B=\B(\B)$. However, as we argue later, $\B$ is not the minimal hypergraph $H$ such that $\B=\B(H)$ as long as $\B\neq \emptyset$.

Starting from a building set $\B$ on $V$ we can associate a polytope which captures its combinatorial information. Define the \emph{$\B$-nestohedron} as the convex polytope $P_{\B}$ in $\mathbb{R}^{V}$ defined by the Minkowsky sum
\begin{equation*}\label{equation:minkowski_sum}
    P_{\B}:=\sum_{I \in \B}\Delta_I,
\end{equation*}
where $\Delta_I\subseteq \R^{V}$ is the convex hull of the standard basis vectors $\{\pmb{e}_v : v \in I\}$---known as a  \emph{standard simplex}. For a hypergraph $H$, we define $P_H:=P_{\B(H)}$\footnote{This polytope is combinatorially and normally equivalent to the definition of the \emph{hypergraph polytope} associated to $H$ in  \cite{DosenPetric2011}.}. Nestohedra were introduced independently by several authors, including Postnikov \cite{Postnikov2009} and Feichtner and Sturmfels \cite{FeichtnerSturmfels2005}. They are dual to the nested set complexes introduced by De Concini and Procesi in \cite{DeConciniProcesi1995} (see also \cite{Zelevinsky2006}). The particular subfamily of graphical nestohedra was studied by Carr and Devadoss under the name of graph-associahedra \cite{CarrDevadoss2006}. Figure \ref{figure:B_nestohedron} illustrates the $\B_0$-nestohedron with building set \begin{equation}\label{equation:B0}
\B_0=\{\{1\},\{2\},\{3\},\{4\},\{1,3\},\{2,3\},\{1,2,3\},\{1,2,4\},\{1,2,3,4\}\}.\end{equation} This polytope is $3$-dimensional and all its vertices lie in the hyperplane $x_1+x_2+x_3+x_4=5$.

\begin{figure}
    \centering 
    \resizebox{0.5\linewidth}{!}{\usetikzlibrary{calc}
\begin{tikzpicture}[
  x=1cm, y=1cm,
  facet/.style   = {fill=black, fill opacity=0.1, draw=none},
  visible/.style = {draw=black!88, line width=0.9pt, line cap=round},
  hidden/.style  = {draw=black!55, line width=0.6pt, dash pattern=on 2.4pt off 2.2pt, line cap=round},
  node4/.style   = {circle, fill=orange!75!red, draw=white, line width=0.4pt, inner sep=0pt, minimum size=3.8pt},
  lbl/.style     = {font=\tiny\ttfamily, text=black!70, inner sep=1.6pt}
]

% ---- projected vertices -------------------------------------------------
\coordinate (v1) at (1.7097, -0.2273); % (0,0,3,2)
\coordinate (v2) at (2.4402, -2.6686); % (0,0,4,1)
\coordinate (v3) at (0.5938, 1.6141); % (4,1,0,0)
\coordinate (v4) at (0.8927, -4.0523); % (0,1,4,0)
\coordinate (v5) at (-2.6282, 1.3046); % (0,2,1,2)
\coordinate (v6) at (2.2413, 1.0794); % (4,0,1,0)
\coordinate (v7) at (-0.5116, 3.2122); % (2,1,0,2)
\coordinate (v8) at (1.3986, 2.6400); % (2,0,1,2)
\coordinate (v9) at (-4.5271, -0.1664); % (1,4,0,0)
\coordinate (v10) at (-4.8526, -1.5593); % (0,4,1,0)
\coordinate (v11) at (-2.4432, 2.6198); % (1,2,0,2)
\coordinate (v12) at (2.8644, -3.1244); % (1,0,4,0)

% ---- front-facing facets, translucent so every edge stays legible ----------------------------------------------
\fill[fill=black, fill opacity=0.1, draw=none] (v4) -- (v2) -- (v1) -- (v5) -- (v10) -- cycle; % 5-gon
\fill[fill=black, fill opacity=0.1, draw=none] (v6) -- (v8) -- (v1) -- (v2) -- (v12) -- cycle; % 5-gon
\fill[fill=black, fill opacity=0.1, draw=none] (v11) -- (v5) -- (v1) -- (v8) -- (v7) -- cycle; % 5-gon
\fill[fill=black, fill opacity=0.1, draw=none] (v12) -- (v2) -- (v4) -- cycle; % 3-gon
\fill[fill=black, fill opacity=0.1, draw=none] (v10) -- (v5) -- (v11) -- (v9) -- cycle; % 4-gon

% ---- hidden edges, dashed through the translucent facets ----------------------------------------
\draw[hidden] (v3) -- (v9);
\draw[hidden] (v3) -- (v6);
\draw[hidden] (v3) -- (v7);

% ---- visible edges -----------------------------------------------------
\draw[visible] (v2) -- (v4);
\draw[visible] (v1) -- (v2);
\draw[visible] (v1) -- (v5);
\draw[visible] (v5) -- (v10);
\draw[visible] (v4) -- (v10);
\draw[visible] (v6) -- (v8);
\draw[visible] (v1) -- (v8);
\draw[visible] (v2) -- (v12);
\draw[visible] (v6) -- (v12);
\draw[visible] (v5) -- (v11);
\draw[visible] (v7) -- (v8);
\draw[visible] (v7) -- (v11);
\draw[visible] (v4) -- (v12);
\draw[visible] (v9) -- (v10);
\draw[visible] (v9) -- (v11);

% ---- vertices and labels ----------------------------------------------
\node[node4] at (v1) {};
\node[lbl, below right] at (1.7097, -0.2273 + 0.5) {\Large $(0,0,3,2)$};
\node[node4] at (v2) {};
\node[lbl, below right] at (2.4402 + 0.1, -2.6686 + 0.5) {\Large $(0,0,4,1)$};
\node[node4] at (v3) {};
\node[lbl, above right] at (v3) {\Large $(4,1,0,0)$};
\node[node4] at (v4) {};
\node[lbl, below right] at (v4) {\Large $(0,1,4,0)$};
\node[node4] at (v5) {};
\node[lbl, above left] at (v5) {\Large $(0,2,1,2)$};
\node[node4] at (v6) {};
\node[lbl, above right] at (v6) {\Large $(4,0,1,0)$};
\node[node4] at (v7) {};
\node[lbl, above left] at (v7) {\Large $(2,1,0,2)$};
\node[node4] at (v8) {};
\node[lbl, above right] at (v8) {\Large $(2,0,1,2)$};
\node[node4] at (v9) {};
\node[lbl, below left] at (v9) {\Large $(1,4,0,0)$};
\node[node4] at (v10) {};
\node[lbl, below left] at (v10) {\Large $(0,4,1,0)$};
\node[node4] at (v11) {};
\node[lbl, above left] at (v11) {\Large $(1,2,0,2)$};
\node[node4] at (v12) {};
\node[lbl, below right] at (v12) {\Large $(1,0,4,0)$};
\end{tikzpicture}}
    \caption{A $\B$-nestohedron }
    \label{figure:B_nestohedron}
\end{figure}
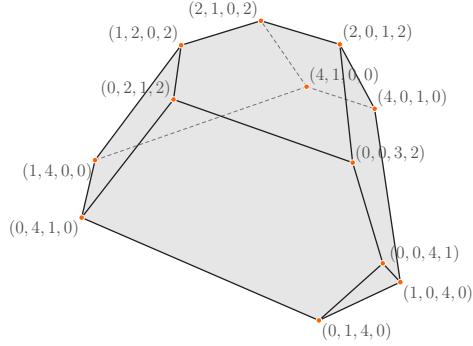

To any polytope $P$ one can define its \emph{$f$-polynomial} by
$$f_P(t):=\sum_{k\ge 0} f_k t^k,$$ where $f_i=f_i(P)$ is the number of $i$-dimensional faces of $P$.  A closely related polynomial is its  \emph{$h$-polynomial} defined using the translation $$h_P(t):=f_P(t-1).$$
We denote by $h_{H}(t):=h_{P_{H}}(t)$ the \emph{$h$-polynomial} of a hypergraph $H$ and by $h_{\B}(t)$ the \emph{$h$-polynomial} of a building set $\B$. Note that by definition $$h_H(t)=h_{\B(H)}(t).$$ In our example,  Figure \ref{figure:B_nestohedron} shows that $f_{P_{\B_0}}(t)=12+18t+8t^2+t^3$, and thus $h_{\B_0}(t)=1+5t+5t^2+t^3$. In this article, we adopt the more general perspective of hypergraphs to study the $h$-polynomials of their associated nestohedra.

The polytopes $P_H$ are all \emph{simple}, and hence their polar duals are \emph{simplicial}. It follows then from a result of Bruggesser and Mani in  \cite{BruggesserMani1971} that their $h$-polynomial has nonnegative integer coefficients. The fact that they are strictly positive and unimodal is a consequence of Stanley's $g$-theorem \cite{Stanley1980}.

Postnikov, Reiner, and Williams proved in \cite{PostnikovReinerWilliams2007} that $h_{\mathcal{B}}(t)$ can be computed as the descent enumerator of a family of rooted forests associated to $\mathcal{B}$, introduced by Postnikov in \cite{Postnikov2009} under the name of $\B$-forests. This gives a combinatorial interpretation to the positive coefficients of $h_{\B}(t)$. It is also shown in \cite{PostnikovReinerWilliams2007}  that for a connected chordal building set $\B$, $h_{\B}(t)$ can be computed as the descent enumerator of a family of permutations associated to $\B$, known as $\B$-permutations.

Using the recursive nature of the family of $\B$-forests, Postnikov derived a recursive characterization of  $f$- and $h$-polynomials of nestohedra.

\begin{proposition}[$h$-version of Theorem 7.11 \cite{Postnikov2009}]	\label{proposition:recurrence_postnikov} The $h$-polynomials associated to building sets are determined by the following recurrence relations:
	\begin{enumerate}
		\item If $\B$ is the unique building set on a singleton, $h_{\B}(t)=1$.
		
		\item If $\B$ has connected components $\B_1,\dots,\B_k$, then
		$$h_\B(t)=h_{\B_1}(t)\cdots h_{\B_k}(t).$$
		
		\item If $\B$ is connected, then 
		$$h_\B(t) = \sum_{I\subsetneq V} (t-1)^{|V|-|I|-1} h_{\B|_{I}}(t),$$ where $\B|_{I}:=\{J\in \B\mid J\subseteq I\}$ is the \textit{restriction} of $\B$ to $I$, and by convention  $h_{\B|_\emptyset}(t)=h_{\emptyset}(t)=1$, for the vacuous case $\B=\emptyset$ on $V=\emptyset$.
	\end{enumerate}
\end{proposition}

Another recursion for the $f_{\B}(t)$ was derived by Zelevinsky \cite[Proposition 4.7]{Zelevinsky2006} using a link decomposition of the nested set complex. The corresponding formula for $h$-polynomials is as follows. 
\begin{proposition}[$h$ version of Proposition 4.7 \cite{Zelevinsky2006}]
Let $\B$ a building set on $V$, then we have that
$$\left (n-\kappa(\B)-(t-1)\frac{d}{dt}\right)h_{\B}(t)=\sum_{I\in \B\setminus\B_{\max}}h_{\B|_{I}}(t)h_{{}_I\backslash\B}(t),$$
where $\B_{\max}$ is the set of maximal blocks of $\B$, and ${}_I\backslash\B=\{J\subseteq V\setminus I\mid J\in \B \text{ or } J\cup I\in \B\}.$\footnote{We use a different notation here for ${}_I\backslash\B$ than the one in \cite{Zelevinsky2006} to avoid confusion with the definition of $\B/_{\A}$ that we use below.}
\end{proposition}

The main goal of this work is to provide three new ways of computing  $h_{\B}(t)$ for any building set $\B$: one as a sum of the M\"obius values on maximal intervals of a partially ordered set (\emph{poset} for short), and a pair of recurrences based on the notion of spanning sub-building sets. All our formulas, which we prove in the general context of hypergraphs, are derived from our novel construction of the \emph{poset of weighted spanning hypergraphs} $\WH(V)$ on $V$, see Definition \ref{definition:poset_of_weighted_hypergraphs}. 

We say that a hypergraph $H'$ is a  \emph{spanning subhypergraph} of a hypergraph $H$ if $V(H')=V(H)=V$ and $H' \subseteq H$. This relation defines an order $H'\leq H$ on the set of hypergraphs on $V$. To define a corresponding order on building sets we rely on a stricter version of the inclusion for hypergraphs. Indeed, we define the relation with respect to the minimal generating subset of a building set under the properties \ref{B1} and \ref{B2}.

\begin{definition}\label{definition:struts}
 For a building set $\B$ we denote by $E(\B)\subset \B$ the set of nonsingleton blocks $S\in \B$, such that whenever $S=I\cup J$ for some $I,J\in \B$ with $I\cap J\neq \emptyset$, we have $I=S$ or $J=S$.   We call the elements of $E(\B)$ the \emph{struts}\footnote{In \cite{DosenPetric2011} the authors call $E(\B)$ a \emph{bare hypergraph}. In this language, struts are the hyperedges of the underlying bare hypergraph.} of $\B$.
\end{definition}

For the building set $\B_0$ in \eqref
{equation:B0} we illustrate $E(\B_0)$ in Figure \ref{figure:struts_and_building_set}.

\begin{figure}[ht]
    \centering
    \captionsetup{justification=centering}
    \resizebox{0.2\linewidth}{!}{\definecolor{color_0}{HTML}{FFFFFF}
\definecolor{color_1}{HTML}{2A2A2A}
\definecolor{color_2}{HTML}{000000}

\begin{tikzpicture}[main_node_g/.style={circle,inner sep=2pt,draw = black, line width=0.1mm}]

\node[main_node_g, fill = color_0,text = color_1,minimum size=0.802em] (0) at (0.13,-0.33) {\tiny \textbf{1}};
\node[main_node_g, fill = color_0,text = color_1,minimum size=0.802em] (1) at (-0.26,1.26) {\tiny \textbf{2}};
\node[main_node_g, fill = color_0,text = color_1,minimum size=0.802em] (2) at (-2.25,0.46) {\tiny \textbf{3}};
\node[main_node_g, fill = color_0,text = color_1,minimum size=0.802em] (3) at (-0.66,-1.92) {\tiny \textbf{4}};

\begin{pgfonlayer}{background}
\fill[fill = color_2, opacity = 0.2, line width = 0.1mm, draw = color_2, solid]
(-1.03,-1.90)--(-1.03,-1.92)--(-1.03,-1.94)--(-1.03,-1.97)--(-1.02,-1.99)--(-1.02,-2.01)--(-1.01,-2.03)--(-1.00,-2.05)--(-0.99,-2.08)--(-0.98,-2.10)--(-0.97,-2.12)--(-0.96,-2.13)--(-0.94,-2.15)--(-0.93,-2.17)--(-0.91,-2.19)--(-0.90,-2.20)--(-0.88,-2.22)--(-0.86,-2.23)--(-0.84,-2.24)--(-0.82,-2.25)--(-0.80,-2.26)--(-0.78,-2.27)--(-0.75,-2.27)--(-0.73,-2.28)--(-0.71,-2.28)--(-0.68,-2.29)--(-0.66,-2.29)--(-0.64,-2.29)--(-0.62,-2.28)--(-0.59,-2.28)--(-0.57,-2.27)--(-0.55,-2.27)--(-0.53,-2.26)--(-0.51,-2.25)--(-0.48,-2.24)--(-0.47,-2.23)--(-0.45,-2.22)--(-0.43,-2.20)--(-0.41,-2.19)--(-0.39,-2.17)--(-0.38,-2.15)--(-0.36,-2.13)--(-0.35,-2.12)--(-0.34,-2.10)--(-0.33,-2.08)--(0.46,-0.49)--(0.47,-0.47)--(0.48,-0.44)--(0.49,-0.42)--(0.49,-0.40)--(0.50,-0.38)--(0.50,-0.35)--(0.50,-0.33)--(0.50,-0.31)--(0.50,-0.28)--(0.49,-0.26)--(0.49,-0.24)--(0.09,1.35)--(0.08,1.37)--(0.08,1.39)--(0.07,1.41)--(0.06,1.43)--(0.05,1.45)--(0.03,1.47)--(0.02,1.49)--(0.00,1.51)--(-0.01,1.53)--(-0.03,1.54)--(-0.05,1.55)--(-0.07,1.57)--(-0.09,1.58)--(-0.11,1.59)--(-0.13,1.60)--(-0.15,1.61)--(-0.17,1.61)--(-0.20,1.62)--(-0.22,1.62)--(-0.24,1.62)--(-0.26,1.62)--(-0.29,1.62)--(-0.31,1.62)--(-0.33,1.62)--(-0.36,1.61)--(-0.38,1.61)--(-0.40,1.60)--(-0.42,1.59)--(-0.44,1.58)--(-0.46,1.57)--(-0.48,1.55)--(-0.50,1.54)--(-0.52,1.53)--(-0.53,1.51)--(-0.55,1.49)--(-0.56,1.47)--(-0.57,1.45)--(-0.59,1.43)--(-0.60,1.41)--(-0.61,1.39)--(-0.61,1.37)--(-0.62,1.35)--(-0.63,1.33)--(-0.63,1.30)--(-1.03,-1.87)--cycle;
\end{pgfonlayer}

\path[draw  = color_2 ,thick, solid](0) edge node {} (2);

\path[draw  = color_2 ,thick, solid](1) edge node {} (2);
\end{tikzpicture}}
    \caption{$E(\B_0)=\{\{1,3\},\{2,3\},\{1,2,4\}\}$.}
    \label{figure:struts_and_building_set}
\end{figure}
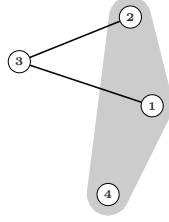  

The sets $\B$ and $E(\B)$ determine each other uniquely on $V(\B)$, and hence we will  consider $\B$ and the pair $(V(\B),E(\B))$ indistinctly as the same object. This notation is inspired on the case of the graphical building set $\B(G)$ of a simple graph $G$ where the struts are precisely the edges of $G$, i.e., $E(G)=E(\B(G))$. For a general hypergraph we define $E(H):=E(\B(H))$. A consequence of this definition is that $\B(E(H))=\B(H)$. Furthermore, $P_{E(\B)}=P_{\B}$ and $E(\B)$ is the minimal hypergraph with this property.

We define spanning sub-building sets as the natural generalization of spanning subgraphs of a simple graph, as follows.  

\begin{definition}
We say that a building set $\A$ is a \emph{spanning sub-building set} of $\B$ if $E(\A)$ is a spanning subhypergraph of $E(\B)$, i.e., if $V(\A)=V(\B)$ and $E(\A) \subseteq E(\B)$.    
\end{definition}

\begin{remark} A less strict, and perhaps more natural, notion of a spanning sub-building set can be defined whenever $\A\subseteq \B$. However, this notion differs from the one considered in this article. In particular, as we will show, a version of Theorem \ref{Theorem:second_recurrence_h} below is true under the two notions, but the stricter condition results in recursive formulas with the least number of terms (see Proposition \ref{proposition:second_recurrence_h_hypergraphs}).
\end{remark}

We consider the order relation  $\A\leq \B$ whenever $\A$ is a spanning sub-building set of $\B$. Figure \ref{figure:spanning_building_set} shows the poset of building sets that lie below $\B_0$ from Figure \ref{figure:struts_and_building_set}. This poset is isomorphic to the Boolean algebra $\mathbb{B}_{E(\B)}$ on $E(\B)$.

\begin{figure}[ht]
    \centering
    \captionsetup{justification=centering}
    \resizebox{0.6\linewidth}{!}{\input{Figures/building_sets_poset}}
    \caption{The spanning sub-building sets of $\B_0$.}
    \label{figure:spanning_building_set}
\end{figure}

We show that the M\"obius values of the maximal intervals determined by weighted hypergraphs supported on $H$, in the poset $\WH(V)$ of Definition \ref{definition:poset_of_weighted_hypergraphs}, give---up to sign---the coefficients of $h_H(t)$.

\begin{theorem}\label{theorem:h_polynomial_as_mobius_values} Let $H$ be a hypergraph on $V$. Then 
\begin{equation}\label{equation:h_polynomial_as_mobius_values_hypergraph}
   h_{H}(t)=(-1)^{|H|}\sum_{\pmb{H}\in \WH(V)}\mu(\pmb{\hat{0}},\pmb{H})t^{w_{\pmb{H}}},
\end{equation}
where the sum is taken over all weighted hypergraphs  $\pmb{H}=(H,w_{\pmb{H}})$  supported on $H$ and  $t^{w_{\pmb{H}}}:=\prod_{I\in H_{\max}}t^{w_{\pmb{H}}(I)}$, with $H_{\max}$ the set of maximal subsets of $V$ connected by $H$. 

In particular,
\begin{equation*}
   h_{\B}(t)=(-1)^{|E(\B)|}\sum_{\pmb{E(\B)}\in \WH(V)}\mu(\pmb{\hat{0}},\pmb{E(\B)})t^{w_{\pmb{E(\B)}}}. 
\end{equation*}
\end{theorem}

In Figure \ref{figure:spanning_weighted_building_set} we show the lower order ideal in $\WH([4])$ generated by the weighted hypergraphs supported by $E(\B_0)$ in Figure \ref{figure:struts_and_building_set}. It can be checked using the recursive definitions of the M\"obius function (see equations \eqref{equation:mobius_recursion_1} and \eqref{equation:mobius_recursion_2}) that the values $\mu(\pmb{\hat{0}},E(\B_0)^i)$, for $i=0,1,2,3$, are $-1$, $-5$, $-5$, and $-1$, respectively. They coincide with the coefficients of $(-1)^{|E(\B_0)|}h_{\B_0}(t)=-h_{\B_0}(t)$ above.

Finally, we use Theorem \ref{theorem:h_polynomial_as_mobius_values} to prove that the $h$-polynomials of nestohedra can be computed using any of the next two recursions. 
\begin{theorem}	\label{Theorem:second_recurrence_h} The $h$-polynomials associated to building sets are determined by the following recurrence relations:
	\begin{enumerate}
		\item If $\B$ consists of a singleton, $h_{\B}(t)=1$.
		
		\item If $\B$ has connected components $\B_1,\dots,\B_k$, then
		$$h_\B(t)=h_{\B_1}(t)\cdots h_{\B_k}(t).$$

        \item If $\B$ is a connected building set with $|V(\B)|> 1$, then $h_{\B}(t)$ satisfies the two recurrences:
    \begin{align}\label{equation:hrecursion1}
    h_\B(t) = \sum_{\A< \B} (-1)^{|E(\B)|-|E(\A)|-1}[\kappa(\A)]_t h_{\A}(t),
\end{align}
and
\begin{align}\label{equation:hrecursion2}
    h_{\B}(t) = \sum_{\substack{\A \le \B\\ E(\A)\neq \emptyset}}(-1)^ {|E(\A)|-1} h_{\B/_\A}(t)\prod_{I\in \A_{\max}} [\,|I|\,]_t,
\end{align}
        where ${\B/_\A}$ is the contraction of the building set $\B$ by $\A$ (see Definition \ref{definition:building_set_contraction}) and $[n]_t:=1+t+\cdots +t^{n-1}$.
	\end{enumerate}
\end{theorem}

\begin{remark} Comparing the recursion of Proposition \ref{proposition:recurrence_postnikov} with the new recursions of Theorem \ref{Theorem:second_recurrence_h}, the reader may notice that they are fundamentally different and, in particular, involve a different number of terms. The latter requires fewer terms to be computed when applied to building sets $\B$ satisfying $|E(\B)|<|V(\B)|$.
\end{remark}

\begin{remark}
    A key application of \eqref{equation:hrecursion1} appears in  \cite{AvilaCarrilloGonzalezDLeon2026} where we use this recurrence to extend and prove a conjectural relation, posed by the third author and Wachs in \cite{GonzalezDLeonWachs2026}, between a new polynomial invariant of a graph (called the $\mu$-polynomial) and the $h$-polynomial of its nestohedron.
\end{remark}

This article is organized as follows: In Section \ref{section:poset_of_weighted_hypergraphs} we introduce the poset $\WH(V)$ of weighted hypergraphs on vertex set $V$, describe some of its structural properties and prove some relations satisfied by its M\"obius function. The proofs of Theorems \ref{theorem:h_polynomial_as_mobius_values} and \ref{Theorem:second_recurrence_h} in Section \ref{section:proofs_of_main_results} rely on the concept of a $\B$-forest, which we extend to define $H$-forests and discuss some of its equivalent definitions in Section \ref{section:B_forests}. The main results also depend on the auxiliary Proposition \ref{proposition:main_argument} whose proof relies on a judicious sign-reversing involution, which we construct in Section \ref{section:sign_reversing_involution}. We conclude in Section \ref{section:examples} with two examples using  Theorem \ref{Theorem:second_recurrence_h} to compute the $h$-polynomials of two infinite families of nestohedra.

\section{The poset of weighted hypergraphs}\label{section:poset_of_weighted_hypergraphs}

In this section we introduce the poset of weighted hypergraphs and derive some properties of its M\"obius function. Moreover, we study the auxiliary polynomial $M_H(t)$ given by the left-hand side of \eqref{equation:h_polynomial_as_mobius_values_hypergraph} and develop its main properties. This will be essential in the proof of Theorem \ref{theorem:h_polynomial_as_mobius_values}. We refer the reader to \cite{Rota1964, Wachs2007} for undefined terminology and results concerning posets.

 \subsection{Hypergraphs and connectivity}\label{section:hypergraphs_and_connectivity}
 
 Formally, a hypergraph $H$ on $V=:V(H)$ is a collection of pairs $\lset{I}=(I,i)\in H$, called \emph{hyperedges}, where $\emptyset\neq I\subseteq V$ and $i$ is any label that makes the pair unique in $H$.  We consider two hyperedges $(I,i)$ and $(J,j)$ as different if either $I\neq J$ or $i\neq j$, and we denote by $V(\lset{I}):=I$ the \emph{underlying vertex set} of $\lset{I}$. If there is a single hyperedge $\lset{I}=(I,i)$ with vertex set $I\subseteq V(H)$, we can abuse notation and ignore its label,  denoting such hyperedge simply by $I$.  Building sets $\B$, and more generally  hypergraphs without repeated hyperedges, can be considered as hypergraphs---$I\in \B$ can, for example, be viewed as the hyperedge $(I,I)$, with underlying set $I$ and label $I$.    
 
    For any hypergraph $H$ on $V$ and $\emptyset \neq I\subseteq V$ we define the \emph{restriction}  of $H$ to $I$ as the hypergraph on vertex set $I$ and hyperedges  given by
    $$H|_{I}=\{\lset{J}\mid V(\lset{J})\subseteq I\}.$$ We say that $I\subseteq V$ is \emph{connected} by $H$ if there is no nontrivial bipartition (\emph{separation}) $I=A\sqcup B$ such that $H|_{I}=H|_{A}\sqcup H|_{B}$. The set $H_{\max}$ consists of the inclusion-maximal subsets of $V$ connected by $H$. The elements in $H_{\max}$ are pairwise disjoint and form a partition of $V$. The \textit{connected components} of $H$ are the restrictions $H|_{I}$, for each $I\in H_{\max}$. We write $\kappa(H)=|H_{\max}|$ and we say that $H$ is \emph{connected} whenever $H_{\max}=\{V\}$---and hence $\kappa(H)=1$.

    The collection
    \begin{equation*}\label{equation:building_set_of_H}
      \B(H)=\{I\subseteq V\mid H|_I \text{ is connected}\}
    \end{equation*}
     forms a building set and, as we mentioned in Section \ref{section:introduction}, every building set $\B$ is of this form.

\subsection{Contraction and dispensable hyperedges}\label{section:contraction_and_dispensable_elements}

We define the contraction of a hypergraph by a spanning subhypergraph.

\begin{definition}\label{definition:contraction_of_subhypergraphs}
    Given $H' \leq H$ we define the \emph{contraction} $H/_{H'}$ as the hypergraph on vertex set $V(H/_{H'})=H'_{\max}$ and hyperedges given by
\begin{equation}\label{equation:definition:contraction_of_hypergraphs}
   H/_{H'}=\bigg \{\bigg(\{I \in H'_{\max} \mid I\cap V(\lset{J})\neq \emptyset\},\lset{J}\bigg)\,\bigg |\, \lset{J}\in H\setminus H'\bigg \}.
\end{equation} In particular,  $|H/_{H'}|=|H|-|H|'$.
\end{definition}

Note that the restriction operation is defined similarly for hypergraphs and for building sets, since both families are closed under the same restriction. In contrast, these families are not closed under \eqref{equation:definition:contraction_of_hypergraphs}.  Indeed, even when $H$ does not have multiple hyperedges, singleton hyperedges, or non-strut hyperedges, all three cases may appear in a contraction $H/_{H'}$. Therefore contraction must be defined slightly differently for building sets.

We say that two hypergraphs $H$ and $H'$ are isomorphic if there is a bijection $f:V(H)\to V(H')$ such that for every $I\subseteq V$, we have that  
$$|\{\lset{I}\in H \mid V(\lset{I})=I\}|=|\{\lset{J}\in H' \mid V(\lset{J})=f(I)\}|,$$
i.e., the number of hyperedges in $H$ with underlying vertex set $I$ is the same as the number of hyperedges in $H'$ with underlying vertex set $f(I)$.

Following a similar idea as in \cite{DosenPetric2011}, we say that $\lset{J}\in H$ is \emph{dispensable} in $H$ when $H|_{V(\lset{J})}\setminus \lset{J}$ is connected. In the language of \cite{DosenPetric2011}, a \emph{bare hypergraph} is a hypergraph without dispensable hyperedges. Note that for a hypergraph $H$, $E(H)$ coincides with any representative of the isomorphism class of the hypergraphs obtained after a one-by-one removal of dispensable elements  at each step. It follows that the resulting hypergraphs are all isomorphic,  independently of the order of the dispensable hyperedges removed at each step. 
 In particular, their hyperedges can be regarded as having no labels.

\begin{definition}\label{definition:building_set_contraction}
For building sets $\A$ and $\B$ such that $\A \le \B$, i.e.,  $E(\A)\subseteq E(\B)$, we define \textit{the contracted building set} $\B/_{\A}$ as the building set associated to the contracted hypergraph $\B/_\A$, with $\B$ and $\A$ regarded as hypergraphs.
\end{definition}

A consequence of Definition \ref{definition:building_set_contraction} is that  the struts of $\B/_{\A}$ can be obtained from either  hypergraph contractions $\B/_{\A}$ or $E(\B)/_{E(\A)}$, after removing all dispensable hyperedges. The following example illustrates Definitions \ref{definition:contraction_of_subhypergraphs} and  \ref{definition:building_set_contraction}.

\begin{example} Consider the two hypergraphs on $[9]$
\begin{align*}
    H=\{&\{1,3\}, \{3,5\}, \{5,6\}, \{7,8\}, \{1,2,3\}, \{1,5,6\},\\
   &\{4,5,6\}, \{7,8,9\}, \{4,7,8,9\}, \{3,4,5,7,8,9\}\}\\
    H'=\{&\{1,5,6\}, \{4,7,8,9\}, \{7,8\}\}.
\end{align*}
Note that $H'\le H$ and both of them are bare hypergraphs.
 The contraction $H/_{H'}$ has vertex set $$H'_{\max} = \{ I_1 := \{1,5,6\}, I_2: = \{2\}, I_3 := \{3\}, I_4 := \{4,7,8,9\} \} = V(H/_{H'}),$$ and the seven hyperedges \begin{align*}
&(\{I_1,I_3\},\{1,3\}), (\{I_1,I_3\},\{3,5\}) \text{ with the same underlying vertex set},\\
&(\{I_1\},\{5,6\}), (\{I_4\},\{7,8,9\}) \text{ which are singletons},\\
&(\{I_1,I_2,I_3\}, \{1,2,3\}),\, (\{I_1,I_4\},\{4,5,6\}), \text{ and }\\
&(\{I_1,I_3,I_4\},\{3,4,5,7,8,9\}), \text{ which is not a strut}.
\end{align*} Removing the dispensable hyperedges we obtain the struts 
    $$E(H/_{H'}) = \{ \{I_1,I_3\}, \{I_1,I_4\}, \{I_1,I_2,I_3\} \}$$ of the building set $\B(H)/_{\B(H')}$, see Figure \ref{figure:contraction_hypergraph}.

    \begin{figure}[ht]
        \centering 
        \resizebox{\linewidth}{!}{\input{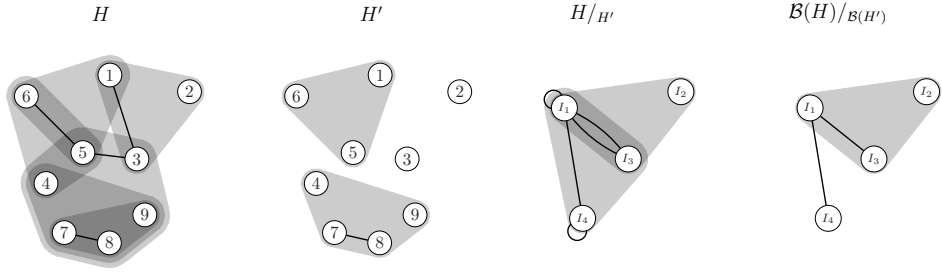}}
        \caption{Contraction of hypergraphs and building sets.}
        \label{figure:contraction_hypergraph}
    \end{figure}

\end{example}

\subsection{Weighted hypergraphs}\label{section:WH}

A \emph{weighted hypergraph} on $V$ is a pair $\pmb{H}=(H,w_{\pmb{H}})$ where $H$ is a hypergraph on $V$ and $w_{\pmb{H}}:H_{\max}\rightarrow \N$ is a weight function that assigns a value $0\le w_{\pmb{H}}(I) <|I|$, for every $I\in H_{\max}$.

\begin{definition}\label{definition:poset_of_weighted_hypergraphs}
The poset of weighted hypergraphs $\WH(V)$ is defined by relations $\pmb{H'}\le \pmb{H}$ whenever $H' \leq H$ and for every $J\in H_{\max}$ of the form $J=I_1\cup\cdots \cup I_k$ where $I_1,\dots,I_k\in H'_{\max}$, then 
\begin{equation}\label{equation:definition_poset_weighted_hypergraphs}
  w_{\pmb{H}}(J)=\mu+\sum_{j=1}^kw_{\pmb{H}'}(I_j),\qquad 0\le \mu<k.  
\end{equation}
\end{definition}
 
The poset $\WH(V)$ has a unique minimal element $\pmb{\hat{0}}=(\emptyset,w_{\hat{0}})$, which is formed by the empty hypergraph $\emptyset$ on $V$, where every point of $V$ has weight $0$. 

For a hypergraph $H$ on $V$ we denote by $\WH(H)$ the lower ideal generated by weighted hypergraphs of the form $\pmb{H}=(H,w_{\pmb{H}})$, i.e., $\WH(H)$ has maximal elements of the form $\pmb{H}=(H,w_{\pmb{H}})$ where $w_{\pmb{H}}$ is a valid weight function of $H$. Note that $H' \leq H$ implies $\WH(H')\subseteq \WH(H)$. If $\
H$ is connected ($H_{\max}=\{V\}$), we denote the maximal elements in $\WH(H)$ by $H^0,H^1,\dots, H^{|V|-1}$, where $H^i=(H,w_{{H^i}})$ and $w_{{H^i}}(V)=i$. Figure \ref{figure:spanning_weighted_building_set} gives the poset $\WH(E(\B_0))$, where $\B_0$ is the building set in Figure \ref{figure:struts_and_building_set}.

\begin{sidewaysfigure}

\centering
  \input{Figures/weighted_hypergraphs_poset}
   \caption{The poset $\WH(E(\B_0))$.}
\label{figure:spanning_weighted_building_set}
\end{sidewaysfigure}

The proof of the following lemma is left to the reader. It is a  a direct application of the definition of $\WH(H)$ together with the definitions of restriction and contraction of hypergraphs.

\begin{lemma}\label{lemma:intervals_defined_by_an_element}
Given $\pmb{H'}\in \WH(H)$, we have the  isomorphisms of posets
    \begin{align}
            U_H(\pmb{H'})& \cong \WH(H/_{H'}),\label{equation:lemma_intervals_eq1}\\
[\pmb{\hat{0}},\pmb{H'}]&\cong \prod_{I\in H'_{\max}}\left [\pmb{\hat{0}},H'|_{I}^{w_{\pmb{H'}}(I)}\right ],\label{equation:lemma_intervals_eq2}
    \end{align}
    where $U_H(\pmb{H'})$ is the principal filter---or principal upper order ideal---defined by $\pmb{H'}$ in $\WH(H)$ and $[\pmb{\hat{0}},\pmb{H'}]$ indicates the closed interval between $\pmb{\hat{0}}$  and  $\pmb{H'}$ in $\WH(H)$.  
\end{lemma}

\subsection{The M\"obius function in \texorpdfstring{$\WH(V)$}{}}\label{section:Mobius}
Recall that the \emph{M\"obius function} of a poset $P$, see e.g.  \cite{Rota1964, Wachs2007}, denoted $\mu_P(x,y)$ (or  $\mu(x,y)$ when the poset is understood), is defined recursively on the set $\textup{Int}(P)$ of closed intervals $[x,y]$ in $P$ by $\mu(x,x)=1$ and satisfying for all $x<y$ either one of the following two equivalent recursions:
\begin{align}
    \mu(x,y)&=-\sum_{x\le z<y}\mu(x,z),\text{ or }\label{equation:mobius_recursion_1}\\
    \mu(x,y)&=-\sum_{x< z\le y}\mu(z,y).\label{equation:mobius_recursion_2}
\end{align}

In order to establish the formula in \eqref{equation:h_polynomial_as_mobius_values_hypergraph} for a given hypergraph $H$ on $V$, we introduce the auxiliary polynomial
\begin{equation}\label{equation:definition_M_poly}
  M_{H}(t):=\sum_{\pmb{H}\in 
\WH(V)
}\mu(\pmb{\hat{0}},\pmb{H})t^{w_{\pmb{H}}},
\end{equation}
  where the sum is taken over all weighted hypergraphs  $\pmb{H}=(H,w_{\pmb{H}})$  supported on $H$ and $t^{w_{\pmb{H}}}:=\prod_{I\in H_{\max}}t^{w_{\pmb{H}}(I)}$. Theorem \ref{theorem:h_polynomial_as_mobius_values} states  that $h_{\B(H)}(t)=(-1)^{|H|}M_{H}(t)$ for any hypergraph $H$, and in particular for $H=E(\B)$.

\begin{example} As mentioned after Theorem \ref{theorem:h_polynomial_as_mobius_values}, Figure \ref{figure:spanning_weighted_building_set} gives the M\"obius values $-1,-5,-5,-1$ for  $\mu(\pmb{\hat{0}},E(\B_0)^i)$, $i=0,1,2,3$, respectively. Thus  
$$M_{E(\B_0)}(t)=-1-5t-5t^2-t^3.$$
\end{example}

The following proposition establishes a recursive way to compute the polynomials $M_H(t)$, in the same spirit of Proposition \ref {Theorem:second_recurrence_h} for $h$-polynomials. The proof follows a similar idea as in \cite[Theorem 5.7]{GonzalezDLeonWachs2026}.

\begin{proposition}\label{proposition:recurrenceM} Let $H$ be an hypergraph on $V$.

\begin{enumerate}
\item If $H=\emptyset$, then $M_{H}(t)=1$.

\item If $H$ has connected components $H_1,\dots,H_k$, then $$M_H(t)=M_{H_1}(t)\cdots M_{H_k}(t).$$

\item If $H$ is connected and  $|V(H)|>1$, then  
 \begin{align}
     \sum_{H' \leq H}[\kappa(H')]_tM_{H'}(t)&=0,\label{equation:M_lemma_1}\\
     \sum_{H' \leq H}M_{H/_{H'}}(t)\prod_{I \in H'_{\max}}[|I|]_t&=0.\label{equation:M_lemma_2}
 \end{align} 
\end{enumerate}
\end{proposition}

\begin{proof} (1) follows directly from the value of the M\"obius function on a poset in a single point.  The proof of (2) follows using \eqref{equation:lemma_intervals_eq2} and the multiplicative nature of the M\"obius function (see for example \cite{Wachs2007}).

For (3) we use the recursive definition of the M\"obius function in equations \eqref{equation:mobius_recursion_1} and \eqref{equation:mobius_recursion_2}. Since $H$ is connected, the maximal intervals of $\WH(H)$ are $[\pmb{\hat{0}},H^i], i=0,\dots, |V|-1$. The  recursive definitions of the M\"obius function in these interval take the form 
 \begin{align}
     \sum_{\pmb{H'}\in [\pmb{\hat{0}},H^i]}\mu(\pmb{\hat{0}},\pmb{H'})&=0,\label{equation:recursive_interval_1}\\
     \sum_{\pmb{H'}\in [\pmb{\hat{0}},H^i]}\mu(\pmb{H'},H^i)&=0.\label{equation:recursive_interval_2}
 \end{align}
Note also that iterated applications of \eqref{equation:definition_poset_weighted_hypergraphs} imply that for any given  $\pmb{H'}=(H',w_{\pmb{H'}})$ in $ \WH(H)$  we have that $\pmb{H'}\leq H^i$ if and only if $i=w_{\pmb{H'}}+j$ where $j=0,\dots,\kappa(H')-1$.
To prove \eqref{equation:M_lemma_1},
we combine  \eqref{equation:recursive_interval_1} for all $i=0,\dots, |V|-1$, interchange the order of summation, and use the fact above
to see that 
 \begin{align*}
0 =     \sum_{i=0}^ {|V|-1}t^i\sum_{\pmb{H'}\in [\pmb{\hat{0}},H^i]}\mu(\pmb{\hat{0}},\pmb{H'}) &= \sum_{\pmb{H'}\in \WH(H)}\mu(\pmb{\hat{0}},\pmb{H'})\sum_{j=0}^ {\kappa(H')-1}t^{w_{\pmb{H'}}+j} \\
 &=   \sum_{\pmb{H'}\in \WH(H)}\mu(\pmb{\hat{0}},\pmb{H'})t^{w_{\pmb{H'}}}[\kappa(H')]_t.
 \end{align*} 
 
Sorting the indexes $(H',w_{\pmb{H'}})$ of this sum, first by the spanning subhypergraphs $H'\leq H$, and then by a valid weight $w_{\pmb{H'}}$, we obtain \begin{align*}
0 = \sum_{H'\leq H}
    \sum_{w_{\pmb{H'}}} \mu(\pmb{\hat{0}},\pmb{H'})t^{w_{\pmb{H'}}}[\kappa(H')]_t &= 
     \sum_{H'\leq H}[\kappa(H')]_t \sum_{w_{\pmb{H'}}} \mu(\pmb{\hat{0}},\pmb{H'})t^{w_{\pmb{H'}}} \\
&= \sum_{H'\leq H}[\kappa(H')]_t M_{H'}(t),
 \end{align*} as required.

In a similar fashion, to establish \eqref{equation:M_lemma_2}, we combine  \eqref{equation:recursive_interval_2} for all $i=0,\dots, |V|-1$ to get 
   \begin{align*}   0=\sum_{i=0}^{|V|-1}t^i\sum_{\pmb{H'}\in [\pmb{\hat{0}},H^i]}\mu(\pmb{H'},H^i)=  \sum_{\pmb{H'}\in \WH(H)}t^{w_{\pmb{H'}}}\sum_{j=0}^{\kappa(H')-1} \mu(\pmb{H'},H^{w_{\pmb{H'}}+j})t^{j}.
 \end{align*}

Once again, sorting first by spanning subhypergraphs and then by valid weights, we find  \begin{align*}
0 = \sum_{H'\leq H}
    \sum_{ w_{\pmb{H'}}} t^{w_{\pmb{H'}}}\sum_{j=0}^{\kappa(H')-1}\mu(\pmb{H'},H^{w_{\pmb{H'}}+j})t^{j} &=  \sum_{H'\leq H}
    \sum_{ w_{\pmb{H'}}}  t^{w_{\pmb{H'}}} M_{H/_{H'}}(t) \\
& = \sum_{H'\leq H} M_{H/_{H'}}(t) \sum_{ w_{\pmb{H'}}} t^{w_{\pmb{H'}}}\\
& = \sum_{H'\leq H}M_{H/_{H'}}(t)\prod_{I \in H'_{\max}}[|I|]_t,
 \end{align*} as needed.\end{proof}

\section{On \texorpdfstring{$\B$}{}-forests and \texorpdfstring{$H$}{}-forests}\label{section:B_forests}

It was shown in \cite{PostnikovReinerWilliams2007}
that $h$-polynomials of nestohedra admit a  combinatorial interpretation in terms of a family of rooted forests introduced in \cite{Postnikov2009}, known as $\B$-forests. In this section, we recall this notion in the language of hypergraphs. Recall that the collection of subsets of $V$ that are connected in a hypergraph $H$ on $V$ forms a building set $\B(H)$, and that this map is not injective. For a hypergraph $H$ the notion of $H$-forest below coincides with the notion of $\B(H)$-forest in \cite{Postnikov2009}. These objects are also in bijective correspondence with \textit{maximal nested sets} in \cite{DeConciniProcesi1995} and with \textit{constructions} in \cite{DosenPetric2011}.

Recall that a \emph{rooted forest} is a forest in which a distinguished node (a \emph{root}) is selected in each connected component (tree). For a rooted forest $F$, given two nodes $u$ and $v$ of $F$, we say that $u$ is a \textit{descendant} of $v$ if $v$ belongs to the shortest path connecting $u$ with the root of its connected component. We denote the set of descendants of $v$ in $F$ by $F_{\leq v}$. Note that  $v\in F_{\leq v}$. Moreover, if $F$ has roots $r_1,\dots,r_M$, the sets of vertices of the connected components of $F$ form the partition $F_{\max}:=\{F_{\leq_{r_1}},\dots,F_{\leq_{r_M}}\}$ of $V$.

A descendant $u$ of $v$ is called a \textit{child} of $v$ if $u$ is a descendant of $v$ and $(u,v)\in E(F)$. Finally, we say that $u$ and $v$ are \emph{incomparable} in $F$ if neither $u$ is a descendant of $v$, nor $v$ is a descendant of $u$.

\begin{definition}[c.f. Definition 7.7,  \cite{Postnikov2009}] \label{definition:B_forest} 
 A rooted forest $F$ on $V$ is an \textit{$H$-forest} for a hypergraph $H$ on $V$ when 
	\begin{enumerate}[label=(\textbf{HF\arabic*})]

\item\label{HF1} For every $v\in V$, $H|_{F_{\leq v}}$ is connected.
		
\item\label{HF2} For any incomparable nodes $u_1,\dots,u_k$ in $F$, $k\geq 2$,  $H|_{\bigcup_{j=1}^{k} F_{\leq u_j}}$ is not connected.
        
\item\label{HF3} We have that  $F_{\max}=H_{\max}$.

	\end{enumerate}
\end{definition}

The connected components of an $H$-forest are called \textit{$H$-trees}, each of which is associated with a connected component of $H$.

\begin{remark} Given a hypergraph $H$, the previous conditions can be expressed in terms of the corresponding building set $\B(H)$ as follows. \ref{HF1} asks that $F_{\leq v}\in\B(H)$, for all $v\in V$. \ref{HF2} means that $\bigcup_{j=1}^{k} F_{\leq u_j} \not\in \B(H)$, for incomparable nodes $u_1,\dots,u_k$. Finally, \ref{HF3} says that $F_{\max}=\B(H)_{\max}$.

This shows that if $\B=\B(H_1)=\B(H_2)$ for hypergraphs $H_1$ and $H_2$, the $H_1$-forests and $H_2$-forests are the same, and are simply the $\B$-forests defined in \cite{Postnikov2009}. 
\end{remark}

The following is a recursive description of $H$-trees originally presented in \cite{Postnikov2009} in terms of $\B(H)$-trees. 

\begin{proposition}[cf. Proposition 8.5, \cite{PostnikovReinerWilliams2007}] \label{proposition:B_forests}
Consider a connected hypergraph $H$ on $V$, $v\in V$, and let $H_1,\dots,H_k$ be the connected components of the restricted hypergraph $H|_{V\setminus v }$. Then all the $H$-trees $T$ with root $v$ are obtained recursively by connecting to $v$ the roots of $H_i$-trees $T_i$ for $i=1,\dots, k$.
\end{proposition}

Therefore, moving through all the possible values of $v\in V$ in Proposition \ref{proposition:B_forests} we obtain all $H$-trees. When $H$ is not connected, we apply this process to each of its connected components to construct all $H$-forests of $H$.

\begin{example} Consider the hypergraph $H$ and the tree $T$ given in Figure \ref{figure:B_tree_recursively}. By taking $v=1$, $H$ restricted to $\{2,\dots, 8\}$ leaves a hypergraph with three connected components $H_1 = \{(\{2,3\},i), (\{2,3\},ii)\}$, $H_2 = \{\{4,6\}$, $\{6\}\}$ and $H_3 = \{\{5,7\},\{5,7,8\}, (\{7\},a), (\{7\},b)\}$. Removing $v=1$ from $T$ gives a forest with three trees. The reader can verify that these trees correspond to $H_j$-trees for each of $H_j, j=1,2,3$. Therefore, $T$ is an $H$-tree.

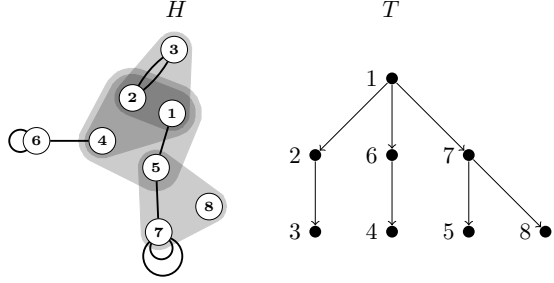
\begin{figure}[ht]
    \centering
    \captionsetup{justification=centering}
    \resizebox{0.6\linewidth}{!}{\definecolor{color_0}{HTML}{FFFFFF}
\definecolor{color_1}{HTML}{000000}
\definecolor{color_2}{HTML}{000000}
\begin{tikzpicture}[main_node_t/.style={circle,fill=black,minimum size=0.5em,inner sep=1pt]}, main_node_g/.style={circle,inner sep=2pt,draw = black, line width=0.1mm}]

\node () at (0,2.5) {$H$};
\node () at (3.2,2.5) {$T$};

\node[main_node_g, fill = color_0,text = color_1,minimum size=0.561em] (0) at (-0.06,0.95) {\tiny \textbf{1}};
\node[main_node_g, fill = color_0,text = color_1,minimum size=0.561em] (1) at (-0.65,1.18) {\tiny \textbf{2}};
\node[main_node_g, fill = color_0,text = color_1,minimum size=0.561em] (2) at (-0.03,1.90) {\tiny \textbf{3}};
\node[main_node_g, fill = color_0,text = color_1,minimum size=0.561em] (3) at (-1.10,0.54) {\tiny \textbf{4}};
\node[main_node_g, fill = color_0,text = color_1,minimum size=0.561em] (4) at (-0.30,0.14) {\tiny \textbf{5}};
\node[main_node_g, fill = color_0,text = color_1,minimum size=0.561em] (5) at (-2.08,0.54) {\tiny \textbf{6}};
\node[main_node_g, fill = color_0,text = color_1,minimum size=0.561em] (6) at (-0.26,-0.82) {\tiny \textbf{7}};
\node[main_node_g, fill = color_0,text = color_1,minimum size=0.561em] (7) at (0.49,-0.44) {\tiny \textbf{8}};

\begin{pgfonlayer}{background}
\fill[fill = color_2, opacity = 0.2, line width = 0.1mm, draw = color_2, solid]
(-0.75,0.94)--(-0.16,0.71)--(-0.14,0.71)--(-0.13,0.70)--(-0.11,0.70)--(-0.10,0.70)--(-0.08,0.70)--(-0.06,0.70)--(-0.05,0.70)--(-0.03,0.70)--(-0.02,0.70)--(0.00,0.70)--(0.02,0.71)--(0.03,0.71)--(0.05,0.72)--(0.06,0.73)--(0.07,0.74)--(0.09,0.74)--(0.10,0.75)--(0.11,0.77)--(0.12,0.78)--(0.13,0.79)--(0.14,0.80)--(0.15,0.81)--(0.16,0.83)--(0.17,0.84)--(0.18,0.86)--(0.18,0.87)--(0.19,0.89)--(0.19,0.90)--(0.19,0.92)--(0.19,0.94)--(0.23,1.89)--(0.23,1.90)--(0.23,1.92)--(0.23,1.94)--(0.22,1.95)--(0.22,1.97)--(0.22,1.98)--(0.21,2.00)--(0.20,2.01)--(0.20,2.03)--(0.19,2.04)--(0.18,2.06)--(0.17,2.07)--(0.16,2.08)--(0.15,2.09)--(0.14,2.10)--(0.12,2.11)--(0.11,2.12)--(0.10,2.13)--(0.08,2.14)--(0.07,2.14)--(0.05,2.15)--(0.04,2.15)--(0.02,2.16)--(0.00,2.16)--(-0.01,2.16)--(-0.03,2.16)--(-0.04,2.16)--(-0.06,2.16)--(-0.08,2.16)--(-0.09,2.15)--(-0.11,2.15)--(-0.12,2.14)--(-0.14,2.14)--(-0.15,2.13)--(-0.17,2.12)--(-0.18,2.11)--(-0.19,2.10)--(-0.20,2.09)--(-0.22,2.08)--(-0.23,2.07)--(-0.85,1.35)--(-0.86,1.33)--(-0.87,1.32)--(-0.88,1.31)--(-0.89,1.29)--(-0.89,1.28)--(-0.90,1.26)--(-0.90,1.25)--(-0.91,1.23)--(-0.91,1.22)--(-0.91,1.20)--(-0.91,1.18)--(-0.91,1.17)--(-0.91,1.15)--(-0.91,1.13)--(-0.90,1.12)--(-0.90,1.10)--(-0.89,1.09)--(-0.89,1.07)--(-0.88,1.06)--(-0.87,1.05)--(-0.86,1.03)--(-0.85,1.02)--(-0.84,1.01)--(-0.83,1.00)--(-0.82,0.98)--(-0.81,0.98)--(-0.79,0.97)--(-0.78,0.96)--(-0.76,0.95)--cycle;
\end{pgfonlayer}

\begin{pgfonlayer}{background}
\fill[fill = color_2, opacity = 0.2, line width = 0.1mm, draw = color_2, solid]
(0.06,1.28)--(-0.53,1.51)--(-0.55,1.51)--(-0.57,1.52)--(-0.59,1.52)--(-0.61,1.53)--(-0.63,1.53)--(-0.65,1.53)--(-0.68,1.53)--(-0.70,1.53)--(-0.72,1.52)--(-0.74,1.52)--(-0.76,1.51)--(-0.78,1.51)--(-0.80,1.50)--(-0.82,1.49)--(-0.84,1.48)--(-0.86,1.46)--(-0.88,1.45)--(-0.89,1.44)--(-0.91,1.42)--(-0.92,1.40)--(-0.94,1.39)--(-0.95,1.37)--(-0.96,1.35)--(-1.32,0.67)--(-1.33,0.65)--(-1.34,0.64)--(-1.34,0.62)--(-1.35,0.61)--(-1.35,0.59)--(-1.35,0.57)--(-1.36,0.56)--(-1.36,0.54)--(-1.36,0.53)--(-1.35,0.51)--(-1.35,0.49)--(-1.35,0.48)--(-1.34,0.46)--(-1.34,0.45)--(-1.33,0.43)--(-1.32,0.42)--(-1.32,0.40)--(-1.31,0.39)--(-1.30,0.38)--(-1.29,0.37)--(-1.28,0.36)--(-1.26,0.34)--(-1.25,0.33)--(-1.24,0.33)--(-1.22,0.32)--(-1.21,0.31)--(-1.19,0.30)--(-1.18,0.30)--(-1.16,0.29)--(-1.15,0.29)--(-1.13,0.29)--(-1.12,0.29)--(-1.10,0.29)--(-1.08,0.29)--(-1.07,0.29)--(-1.05,0.29)--(-1.04,0.29)--(-1.02,0.30)--(0.04,0.62)--(0.06,0.63)--(0.08,0.64)--(0.10,0.65)--(0.12,0.66)--(0.14,0.67)--(0.16,0.68)--(0.17,0.70)--(0.19,0.71)--(0.20,0.73)--(0.22,0.75)--(0.23,0.77)--(0.24,0.78)--(0.25,0.80)--(0.26,0.82)--(0.27,0.85)--(0.27,0.87)--(0.28,0.89)--(0.28,0.91)--(0.28,0.93)--(0.28,0.95)--(0.28,0.97)--(0.28,1.00)--(0.28,1.02)--(0.27,1.04)--(0.27,1.06)--(0.26,1.08)--(0.25,1.10)--(0.24,1.12)--(0.23,1.14)--(0.22,1.16)--(0.20,1.17)--(0.19,1.19)--(0.17,1.21)--(0.16,1.22)--(0.14,1.23)--(0.12,1.25)--(0.10,1.26)--(0.08,1.27)--cycle;
\end{pgfonlayer}

\begin{pgfonlayer}{background}
\fill[fill = color_2, opacity = 0.2, line width = 0.1mm, draw = color_2, solid]
(-0.56,0.14)--(-0.56,0.13)--(-0.52,-0.84)--(-0.52,-0.86)--(-0.52,-0.87)--(-0.51,-0.89)--(-0.51,-0.90)--(-0.50,-0.92)--(-0.50,-0.93)--(-0.49,-0.95)--(-0.48,-0.96)--(-0.47,-0.97)--(-0.46,-0.99)--(-0.45,-1.00)--(-0.44,-1.01)--(-0.43,-1.02)--(-0.41,-1.03)--(-0.40,-1.04)--(-0.39,-1.05)--(-0.37,-1.06)--(-0.36,-1.06)--(-0.34,-1.07)--(-0.33,-1.07)--(-0.31,-1.08)--(-0.30,-1.08)--(-0.28,-1.08)--(-0.26,-1.08)--(-0.25,-1.08)--(-0.23,-1.08)--(-0.22,-1.08)--(-0.20,-1.07)--(-0.18,-1.07)--(-0.17,-1.06)--(-0.15,-1.06)--(-0.14,-1.05)--(0.61,-0.66)--(0.63,-0.65)--(0.64,-0.64)--(0.65,-0.63)--(0.67,-0.62)--(0.68,-0.61)--(0.69,-0.60)--(0.70,-0.59)--(0.71,-0.57)--(0.71,-0.56)--(0.72,-0.55)--(0.73,-0.53)--(0.73,-0.52)--(0.74,-0.50)--(0.74,-0.48)--(0.74,-0.47)--(0.75,-0.45)--(0.75,-0.44)--(0.75,-0.42)--(0.74,-0.40)--(0.74,-0.39)--(0.74,-0.37)--(0.73,-0.36)--(0.73,-0.34)--(0.72,-0.33)--(0.71,-0.31)--(0.71,-0.30)--(0.70,-0.28)--(0.69,-0.27)--(0.68,-0.26)--(0.67,-0.25)--(0.65,-0.24)--(0.64,-0.23)--(-0.15,0.35)--(-0.17,0.36)--(-0.18,0.37)--(-0.20,0.37)--(-0.21,0.38)--(-0.23,0.39)--(-0.24,0.39)--(-0.26,0.39)--(-0.27,0.40)--(-0.29,0.40)--(-0.30,0.40)--(-0.32,0.40)--(-0.34,0.40)--(-0.35,0.39)--(-0.37,0.39)--(-0.38,0.39)--(-0.40,0.38)--(-0.41,0.37)--(-0.43,0.37)--(-0.44,0.36)--(-0.46,0.35)--(-0.47,0.34)--(-0.48,0.33)--(-0.49,0.32)--(-0.50,0.31)--(-0.51,0.29)--(-0.52,0.28)--(-0.53,0.27)--(-0.54,0.25)--(-0.54,0.24)--(-0.55,0.22)--(-0.55,0.21)--(-0.56,0.19)--(-0.56,0.17)--(-0.56,0.16)--cycle;
\end{pgfonlayer}

\begin{pgfonlayer}{background}
\fill[fill = color_2, opacity = 0.2, line width = 0.1mm, draw = color_2, solid]
(-1.25,0.23)--(-0.45,-0.17)--(-0.43,-0.18)--(-0.41,-0.19)--(-0.39,-0.19)--(-0.37,-0.20)--(-0.35,-0.20)--(-0.33,-0.20)--(-0.30,-0.21)--(-0.28,-0.20)--(-0.26,-0.20)--(-0.24,-0.20)--(-0.22,-0.19)--(-0.20,-0.19)--(-0.18,-0.18)--(-0.16,-0.17)--(-0.14,-0.16)--(-0.12,-0.15)--(-0.10,-0.14)--(-0.08,-0.13)--(-0.07,-0.11)--(-0.05,-0.10)--(-0.04,-0.08)--(-0.02,-0.06)--(-0.01,-0.04)--(0.00,-0.03)--(0.01,-0.01)--(0.02,0.01)--(0.34,0.79)--(0.35,0.82)--(0.36,0.84)--(0.37,0.87)--(0.37,0.90)--(0.37,0.92)--(0.37,0.95)--(0.37,0.98)--(0.37,1.01)--(0.37,1.03)--(0.36,1.06)--(0.35,1.09)--(0.34,1.11)--(0.33,1.14)--(0.32,1.16)--(0.31,1.19)--(0.29,1.21)--(0.27,1.23)--(0.26,1.25)--(0.24,1.27)--(0.22,1.29)--(0.19,1.31)--(0.17,1.32)--(0.15,1.34)--(0.12,1.35)--(0.10,1.36)--(-0.49,1.59)--(-0.52,1.60)--(-0.55,1.61)--(-0.57,1.61)--(-0.60,1.62)--(-0.63,1.62)--(-0.65,1.62)--(-0.68,1.62)--(-0.71,1.62)--(-0.74,1.61)--(-0.76,1.61)--(-0.79,1.60)--(-0.82,1.59)--(-0.84,1.58)--(-0.87,1.57)--(-0.89,1.55)--(-0.91,1.54)--(-0.93,1.52)--(-0.95,1.50)--(-0.97,1.48)--(-0.99,1.46)--(-1.01,1.44)--(-1.02,1.42)--(-1.04,1.39)--(-1.40,0.71)--(-1.41,0.69)--(-1.42,0.67)--(-1.43,0.65)--(-1.44,0.63)--(-1.44,0.61)--(-1.44,0.59)--(-1.45,0.56)--(-1.45,0.54)--(-1.45,0.52)--(-1.44,0.50)--(-1.44,0.48)--(-1.44,0.46)--(-1.43,0.44)--(-1.42,0.41)--(-1.41,0.39)--(-1.40,0.38)--(-1.39,0.36)--(-1.38,0.34)--(-1.37,0.32)--(-1.35,0.30)--(-1.34,0.29)--(-1.32,0.27)--(-1.30,0.26)--(-1.29,0.25)--(-1.27,0.24)--cycle;
\end{pgfonlayer}

\path[draw  = color_2 ,thick, solid](0) edge node {} (4);

\begin{pgfonlayer}{background}
\path[draw = color_2, thick, solid](-0.65,1.18) .. controls (-0.52,1.49) and (-0.31,1.73) .. (-0.03,1.90);
\end{pgfonlayer}

\begin{pgfonlayer}{background}
\path[draw = color_2, thick, solid](-0.65,1.18) .. controls (-0.37,1.36) and (-0.17,1.60) .. (-0.03,1.90);
\end{pgfonlayer}

\path[draw  = color_2 ,thick, solid](3) edge node {} (5);

\path[draw  = color_2 ,thick, solid](4) edge node {} (6);

\begin{pgfonlayer}{background}
\path[draw = color_2, thick, solid](-2.32,0.54) ellipse [x radius=0.166, y radius=0.166];
\end{pgfonlayer}

\begin{pgfonlayer}{background}
\path[draw = color_2, thick, solid](-0.22,-1.06) ellipse [x radius=0.166, y radius=0.166];
\end{pgfonlayer}

\begin{pgfonlayer}{background}
\path[draw = color_2, thick, solid](-0.20,-1.18) ellipse [x radius=0.291, y radius=0.291];
\end{pgfonlayer}

%% -----------------------------------------------------------------------------------------

\node at (2.8142857142857145 + 0.1,1.4714285714285713) {$1$};
\node[main_node_t] (0) at (3.2142857142857144,1.4714285714285713) {};
\node at (1.6714285714285717 + 0.1,0.3285714285714285) {$2$};
\node[main_node_t] (1) at (2.0714285714285716,0.3285714285714285) {};
\node at (2.8142857142857145 + 0.1,0.3285714285714285) {$6$};
\node[main_node_t] (2) at (3.2142857142857144,0.3285714285714285) {};
\node at (3.957142857142857 + 0.1,0.3285714285714285) {$7$};
\node[main_node_t] (3) at (4.357142857142857,0.3285714285714285) {};
\node at (1.6714285714285717 + 0.1,-0.8142857142857142) {$3$};
\node[main_node_t] (4) at (2.0714285714285716,-0.8142857142857142) {};
\node at (2.8142857142857145 + 0.1,-0.8142857142857142) {$4$};
\node[main_node_t] (5) at (3.2142857142857144,-0.8142857142857142) {};
\node at (3.957142857142857 + 0.1,-0.8142857142857142) {$5$};
\node[main_node_t] (6) at (4.357142857142857,-0.8142857142857142) {};
\node at (5.1 + 0.1,-0.8142857142857142) {$8$};
\node[main_node_t] (7) at (5.5,-0.8142857142857142) {};

\draw [->] (0) -- (1);
\draw [->] (0) -- (2);
\draw [->] (0) -- (3);
\draw [->] (1) -- (4);
\draw [->] (2) -- (5);
\draw [->] (3) -- (6);
\draw [->] (3) -- (7);

\end{tikzpicture}}
    \caption{An example of an  $H$-tree $T$.}
    \label{figure:B_tree_recursively}
\end{figure}    
\end{example}

For the arguments below, we restate Proposition \ref{proposition:B_forests}
in the following form.

\begin{lemma}\label{lemma:H_Forest_removing_root}
Let $H$ be a hypergraph and $F$ a rooted forest on $V$. Then $F$ is an $H$-forest if and only if  $F_{\max}=H_{\max}$ and for a root $r$ of $F$,   $F|_{V\setminus r}$ is an $H|_{V\setminus r}$-forest.
\end{lemma}

Given a hypergraph $H$ and $v\in V(H)$, we will denote by  $$H_v=\{\lset{I}\in H \ | \  v \in V(\lset{I})\}\subseteq H,$$ the set of hyperedges of $H$ containing  $v$. 

\begin{remark}\label{remark:Hv_non_empty}Note that $H=\bigcup_{v\in V(H)} H_v$. Moreover, if $|V| > 1$ and $H$ is connected, we see that $H_v \not= \emptyset$ for all $v\in V(H)$. \end{remark}

Furthermore, Proposition \ref{proposition:B_forests} shows that $H$-forests induce a natural decomposition of $H$ as follows.

\begin{lemma}\label{lemma:Btrees_struts}
Let $H$ be a hypergraph on $V$ and let $F$ be an $H$-forest with roots $r_0=r,\dots,r_m\in V$.
\begin{enumerate}
    \item \label{part_1_lemma_Btrees_struts} If $\lset{I}\in H$ and $V(\lset{I})\cap F_{\leq r}\neq \emptyset$, then $V(\lset{I})\subseteq F_{\leq r}$. In particular, $H_r\subseteq H|_{F_{\leq_r}}$, and  $H= H|_{F_{\leq_{r_0}}}\sqcup \dots \sqcup H|_{F_{\leq_{r_m}}}$, as a disjoint union.

    \item If $H$ is connected and  $u_1,\dots,u_l$ are the children of $r$ in $F$, then $$H = H_r\sqcup \bigsqcup_{i=1}^l H|_{F_{\leq_{u_i}}},$$ as a disjoint union.
\end{enumerate}
\end{lemma}

\begin{proof} If $V(\lset{I})\not\subseteq F_{\leq r}$, $V(\lset{I})\cup F_{\leq r}$ is a larger set connected by $H$ contradicting \ref{HF3}, so $V(\lset{I})\subseteq F_{\leq r}$. The other statements follow from this property.
\end{proof}

 For the following result we consider that $V$ is provided with the structure of a total order $V$. For instance, if $V=[n]$ we can use the usual order of the integers. Assuming this order, we can say that an edge $(u,v)\in E(F)$, in an $H$-forest $F$, is a \textit{descent} if $u>v$. The total number of descents in $F$ is  denoted $\des(F)$.

\begin{proposition}[c.f. Corollary 8.4, \cite{PostnikovReinerWilliams2007}]\label{proposition:h_vector_H_forests}
	If $H$ is a hypergraph on an ordered set $V$, then $$h_{H}(t)=\sum_{F\text{ an }H\text{-forest }  } t^{\des(F)}.$$
\end{proposition}

Naturally, if $\B=\B(H)$, for some $H$, this sum over the $H$-forests is the same as the sum over the $\B$-forests, which is consistent with our definition $h_H(t)=h_{\B(H)}(t)$.

\section{A sign reversing involution}\label{section:sign_reversing_involution}

The goal of this section is to establish an analogous equation to \eqref{equation:M_lemma_1} for $h$-polynomials of hypergraphs. More precisely, we prove the following.

\begin{proposition}\label{proposition:main_argument}
Let $H$ be a connected hypergraph such that $|V(H)|>1$. Then \begin{equation}\label{equation:subgraphs}
    \sum_{H' \leq H} [\kappa(H')]_t (-1)^{|H'|}h_{H'}(t)=0.
    \end{equation} 
\end{proposition}

The idea is to rewrite  \eqref{equation:subgraphs}  as a sum over a suitable set of combinatorial objects $\Omega_H$ together with a weight function  $\omega:\Omega_H\rightarrow \mathbb{Z}[t]$. We want $\omega(x)$ to be the contribution of $x\in \Omega_H$ to the left-hand side of \eqref{equation:subgraphs}. Then, we find a sign-reversing involution $\psi:\Omega_H\to \Omega_H,$ i.e., $\psi\circ\psi=\text{id}_{\Omega_H}$ and  $\omega(\psi(x))=-\omega(x)$, and conclude that the left-hand side of \eqref{equation:subgraphs} vanishes.

Based on Proposition \ref{proposition:h_vector_H_forests}, we can replace the polynomial $h_{H'}(t)$, by the sum of $t^{\des(F)}$ over all $H'$-forest $F$. Moreover, if $F$ is an $H'$-forest, we can sort its roots $$r_0 < r_1 < \cdots <r_{\kappa(H')-1}$$ in ascending order, according to the total order of $V$ and define $\sigma_F(r_i):=i$.

Taking these observations into account, we define \begin{align*}
	\Omega_H&:= \left\{(H',F,r)\mid H' \leq H,\, F \text{ an } H' \text{-forest, and }  r \text{ a root of } F \right\},\\
 \omega&:\Omega_H\rightarrow \mathbb{Z}[t],\qquad \omega (H',F,r) =  (-1)^{|H'|}t^{\des(F) + \sigma_F(r)}.
\end{align*} 
Since $[\kappa(H')]_t=1 + t + \cdots + t^{\kappa(H')-1}$, (\ref{equation:subgraphs}) can be rewritten as 
\begin{equation}\label{equation:sum_omega}
\sum_{(H',F,r) \in \Omega_{H}} \omega(H',F,r)=0.
\end{equation} 

In order to establish this equation we need an auxiliary construction. The sign reversing involution we present is based on adding or removing a hyperedge from $H'$ according to a triple $(H',F,r)$, depending on the value of $r$.

\begin{definition}[Adding/removing hyperedges]\label{definition:adding_removing_hyperedges}
If $H' \leq H$ and $\lset{I} \in H$, we denote by $H'{\pm \lset{I}} \subseteq H$ the spanning subhypergraphs \begin{align*}
H'{+\lset{I}} & := H'\cup\{\lset{I}\}, \text{ if } \lset{I}\not\in H',\\ H'{-\lset{I}} &:= H'\setminus\{\lset{I}\}, \text{ if } \lset{I}\in H'.
\end{align*}
\end{definition}
 With a bit more work in Definition \ref{definition:FpmI_forests} below, starting from an $H'$-forest $F$ and a root $r$ of $F$, we will define  $(H'{\pm \lset{I}})$-forests $F^{\pm_r \lset{I}}$. Then, fixing suitable hyperedges $\lset{I}_v\in H_v$, for every $v\in V$---which can be chosen by Remark \ref{remark:Hv_non_empty} and the connectedness of $H$, allow us to define the map 
\begin{equation}\label{equation:sign_reversing_involution_definition}
    \psi: \Omega_{H} \longrightarrow \Omega_{H},\quad \psi(H',F,r)=\begin{cases}
(H'{-\lset{I}_{r}},F^{-_r\lset{I}_r},r),&\text{ if } \lset{I}_{r} \in H',\\
(H'{+\lset{I}_{r}},F^{+_r\lset{I}_r},r),&\text{ if } \lset{I}_r\not\in H'.\end{cases}
\end{equation} 

The remainder of this section will be devoted to proving that $\psi$ is a well-defined sign-reversing involution.

Let $F$ be an $H'$-forest. If $u\in V$, let 
\begin{equation*}
\rho_F(u):=r \textup{  be the root of } F \textup{ such that } u\in F_{\leq r}.
\end{equation*} 

Moreover, if $u\in V$ and $\rho_F(u)=r$, but $u\neq r$, let \begin{equation*}
 \ell_F(u):=v  \textup{ be the child of  } r \textup{ such that } u\in F_{\leq v}.
\end{equation*}

If the context is clear, we simply write $\rho(u)=\rho_F(u)$ and $\ell(u)=\ell_F(u)$. Figure \ref{figure:tree_example} illustrates the use of this notation.

 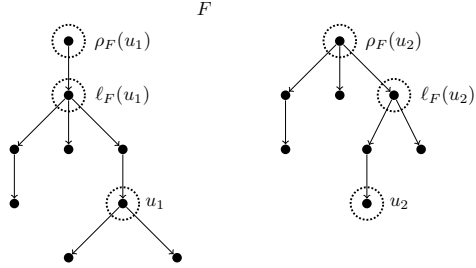
\begin{figure}[ht]
    \centering
    \captionsetup{justification=centering}
    \resizebox{0.5\linewidth}{!}{\begin{tikzpicture}[main_node_t/.style={circle,fill=black,minimum size=0.5em,inner sep=1pt]}, main_node_g/.style={circle,inner sep=2pt,draw = black, line width=0.1mm}]

\node at (2.5,0.6) {$F$};

\node[main_node_t] (0) at (0,0) {};
\node[main_node_t] (1) at (0,-1) {};
\node[main_node_t] (2) at (-1,-2) {};
\node[main_node_t] (3) at (0,-2) {};
\node[main_node_t] (4) at (1,-2) {};
\node[main_node_t] (5) at (-1,-3) {};
\node[main_node_t] (6) at (1,-3) {};
\node[main_node_t] (8) at (0,-4) {};
\node[main_node_t] (9) at (2,-4) {};

\draw [->] (0) -- (1);
\draw [->] (1) -- (2);
\draw [->] (1) -- (3);
\draw [->] (1) -- (4);
\draw [->] (2) -- (5);
\draw [->] (4) -- (6);
\draw [->] (6) -- (8);
\draw [->] (6) -- (9);

\draw [line width=1pt,dash pattern=on 1pt off 1pt] (0,0) circle(0.3cm);
\draw [line width=1pt,dash pattern=on 1pt off 1pt] (0,-1) circle(0.3cm);
\draw [line width=1pt,dash pattern=on 1pt off 1pt] (1,-3) circle(0.3cm);

\node at (1,0) {$\rho_F(u_1)$};
\node at (1,-1) {$\ell_F(u_1)$};
\node at (1.6,-3) {$u_1$};

\node[main_node_t] (0) at (5,0) {};
\node[main_node_t] (1) at (4,-1) {};
\node[main_node_t] (2) at (5,-1) {};
\node[main_node_t] (3) at (6,-1) {};

\node[main_node_t] (4) at (4,-2) {};
\node[main_node_t] (5) at (5.5,-2) {};
\node[main_node_t] (6) at (6.5 ,-2) {};
\node[main_node_t] (7) at (5.5,-3) {};

\draw [line width=1pt,dash pattern=on 1pt off 1pt] (5,0) circle(0.3cm);
\draw [line width=1pt,dash pattern=on 1pt off 1pt] (6,-1) circle(0.3cm);
\draw [line width=1pt,dash pattern=on 1pt off 1pt] (5.5,-3) circle(0.3cm);

\node at (6,0) {$\rho_F(u_2)$};
\node at (7,-1) {$\ell_F(u_2)$};
\node at (6.1,-3) {$u_2$};

\draw [->] (0) -- (1);
\draw [->] (0) -- (2);
\draw [->] (0) -- (3);
\draw [->] (1) -- (4);
\draw [->] (3) -- (5);
\draw [->] (5) -- (7);
\draw [->] (3) -- (6);
\end{tikzpicture}}
	\caption{Nodes $\rho_F(u_i)$ and $\ell_F(u_i)$ for $u_1,u_2\in V$.}
    \label{figure:tree_example}
\end{figure}

\begin{definition}\label{definition:FpmI_forests} Fix $H' \leq H$, $F$ an $H'$-forest, $r$ a root of $F$, and $\lset{I} \in H_r$. We define forests $F^{\pm_r \lset{I}}$ according to the following two cases:
\begin{enumerate}[]

\item If $\lset{I} \in H'$, let $F^{-_r\lset{I}}$ be the forest obtained by removing from $F$ the edges $\{r,\ell(u)\}$, for every $u\in V(\lset{I})\setminus r$ satisfying the following condition:\begin{equation}\label{equation:C_condition}\tag{$C_{H',\lset{I},r}$}
\text{Whenever $\lset{J}\in H'_r$ satisfies that $V(\lset{J})\cap F_{\leq \ell_F(u)}\neq \emptyset$ then $\lset{J}=\lset{I}$.}
\end{equation} In simple terms, we are removing from $F$ the edges $\{r,\ell(u)\}$ such that $H'|_{F_{\leq \ell(u)}}$ becomes a maximal connected component in $H'$ when removing $\lset{I}$ from $H'$. 

\item If $\lset{I} \not\in H'$, let 
 $F^{+_r\lset{I}}$ be the forest obtained from $F$ by attaching the root  $\rho_F(u)$ as a child of $r$ by a new edge, for each $u\in V(\lset{I})\setminus r$ such that $\rho_F(u)\neq r$.

\end{enumerate}

Lemma \ref{lemma: F_pm_I_well_defined} below shows that  $F^{\pm_r \lset{I}}$  are $(H'\pm I)$-forests.   
\end{definition}

\begin{example}\label{example:pm_H_forests} Consider the hypergraphs $H'$ and $H''$ depicted in Figure \ref{figure:F_pm_I}, which are both subhypergraphs of the hypergraph $H$ on $[8]$ given in Figure \ref{figure:B_tree_recursively}. 
We have that $H''=H'{+\lset{I}}$ and $H'=H''{-\lset{I}}$ where $\lset{I}\in H''$ is the unique hyperedge with $V(\lset{I}) = \{1,2,4,5\}$. 
The forest $F'$ and the tree $F''$ verify the conditions in Definition \ref{definition:B_forest}  to be an $H'$-forest and an $H''$-tree, respectively. With respect to  $r=1$, observe first that the children of $1$ in $F''$ are 
$2=\ell_{F''}(2)$, $6=\ell_{F''}(4)$, and $7=\ell_{F''}(5)$, and $(H''-\lset{I})_1=H'_1=\{\lset{J}\}$ where $\lset{J}$ is the unique hyperedge with $V(\lset{J})=\{1,2,3\}$. Since $V(\lset{J})\cap F'_{\le 2}\neq \emptyset$ while $V(\lset{J})\cap F'_{\le 6}= V(\lset{J})\cap F'_{\le 7}= \emptyset$, applying part (1) of Definition \ref{definition:FpmI_forests} we obtain that $F'=F''^{-_1\lset{I}}$. On the other hand, $\rho_{F'}(1)=\rho_{F'}(2)=1$, $\rho_{F'}(4)=6$ and $\rho_{F'}(5)=7$ and so $F''=F'^{+_1\lset{I}}$. This also illustrates that $(F'^{+_1\lset{I}})^{-_1\lset{I}}=F'$ and $(F''^{-_1\lset{I}})^{+_1\lset{I}}=F''$.

\begin{figure}[ht]
    \centering
    \captionsetup{justification=centering}
    \resizebox{1\linewidth}{!}{\input{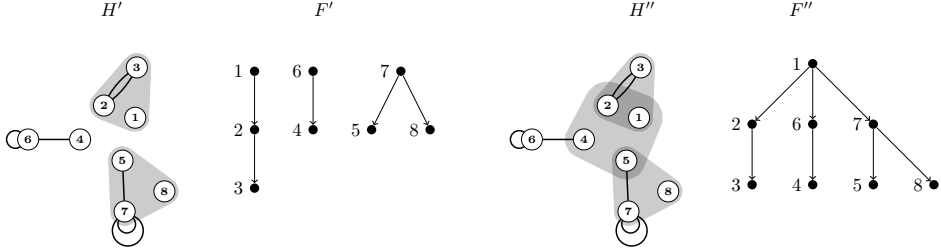}}
	\caption{Example of operations in Definitions \ref{definition:adding_removing_hyperedges} and \ref{definition:FpmI_forests}.}
    \label{figure:F_pm_I}
\end{figure} 
\end{example}

The following lemma contains key properties of the previous constructions.

\begin{lemma}\label{lemma: F_pm_I_well_defined} Fix $(H',F,r)\in\Omega_H$ and $\lset{I}\in H_r$. The following assertions hold:\begin{enumerate}
\item If $\lset{I} \in H'$, then $F^{-_r\lset{I}}$ is an $(H'{-\lset{I}})$-forest.

\item If $\lset{I} \not\in  H'$,  then $F^{+_r\lset{I}}$ is an $(H'{+\lset{I}})$-forest.

\item $(F^{-_r\lset{I}})^{+_r\lset{I}} = (F^{+_r\lset{I}})^{-_r\lset{I}}$=F.
\end{enumerate}	
\end{lemma}

\begin{proof} Let $r=r_0,\dots, r_M$ be a list of the roots of $F$, so $H'_{\max} = \{ F_{\leq r_0}, \dots , F_{\leq r_M}\}$.

In order to prove (1) and (2), we will apply Lemma \ref{lemma:H_Forest_removing_root} by showing first that $(F^{\pm_r \lset{I}})|_{V\setminus r}$ is an $(H'{\pm \lset{I}})|_{V\setminus r}$-forest, and second that $(H'{\pm \lset{I}})_{\max}=(F^{\pm_r \lset{I}})_{\max}$.

For the first condition note that 
\begin{align}
  \label{equation:restriction_part1}(H'{\pm \lset{I}})|_{V\setminus r} &= H'|_{V\setminus r},\\ \label{equation:restriction_part2}(F^{\pm_r \lset{I}})|_{V\setminus r}  &= F|_{V\setminus r}.  
\end{align}
Indeed, \eqref{equation:restriction_part1} holds since $r\in V(\lset{I})$ and the hyperedges of such restrictions are precisely those not containing $r$. To establish  \eqref{equation:restriction_part2}, note that  $F|_{V\setminus r}$ is obtained from $F$ by removing all edges containing $r$ and then Proposition  \ref{proposition:B_forests}. implies that $F|_{V\setminus r}$ is an $H'|_{V\setminus r}$-forest.   We conclude using \eqref{equation:restriction_part1} then that $(F^{\pm_r \lset{I}})|_{V\setminus r}$ is an $(H'{\pm \lset{I}})|_{V\setminus r}$-forest. 

For the second condition, let us write  \begin{equation}\label{equation:different_vertices_of_I}
V(\lset{I})=\{r,u_1,\dots,u_m,v_1,\dots,v_p\},
\end{equation} $m, p\geq 0$, where the $u_i$ satisfy (\ref{equation:C_condition}) and the $v_j$ do not. We analyze the two cases:

\begin{enumerate}
    \item \textbf{If $\lset{I}\in H'$}:  The roots of $F^{-_r\lset{I}}$ are $r, r_1,\dots,r_M,\ell(u_1),\dots,\ell(u_m)$ and \begin{equation}\label{equation:Fmax_minus_I}
  (F^{-_r\lset{I}})_{\max} = \{F_{\leq r_1}, \dots , F_{\leq r_M}\} \cup \{F_{\leq \ell(u_i)}\}_{i=1}^m \cup \{(F^{-_r\lset{I}})_{\leq r }\}.    
  \end{equation}

In order to prove that $(F^{-_r\lset{I}})_{\max}=(H'-\lset{I})_{\max}$ it is sufficient to show that  each $F_{\leq \ell(u_i)}$ and $(F^{-_r\lset{I}})_{\leq r}$ are in $(H'{-\lset{I}})_{\max}$.

For the former, assume by contradiction that there is $\lset{J}\in H'{-\lset{I}}$ with $V(\lset{J})\not\subseteq F_{\leq \ell(u_i)}$ and $V(\lset{J})\cap F_{\leq \ell(u_i)}\neq \emptyset$. If $r\not\in V(\lset{J})$, $\lset{J}$ would be connecting $F_{\leq \ell(u_i)}$ with other $F_{\leq u}$ in $H'$, for some $u\neq r$, $u\in V(\lset{J})\setminus F_{\leq \ell(u_i)}$. But then $F$ does not satisfy the definition of an $H'$-forest, which gives a contradiction. Therefore, $r\in V(\lset{J})$. In turn, this contradicts that $u_i$ satisfies (\ref{equation:C_condition}). We conclude that $F_{\leq \ell(u_i)}$ is maximally connected by $H'-\lset{I}$.

For the latter, note that $$(F^{-_r\lset{I}})_{\leq r} = \{r\}\cup \bigsqcup_{j=1}^p F_{\leq \ell(v_j)}.$$ 

Observe that all $F_{\leq \ell(u_i)}, F_{\leq \ell(v_j)}$ are connected by $H'|_{V \setminus r}$, since they are by $H'$ and none of them contains $r$. Moreover, since each $v_j$ fails to satisfy (\ref{equation:C_condition}), we can find $\lset{J}_j \in H'_r$, $\lset{J}_j\neq \lset{I}$, with $V(\lset{J}_j)\cap F_{\leq \ell(v_j)}\neq \emptyset$. We also have that  $C=\bigcup_{j=1}^p V(\lset{J}_j)$ is connected by  $H'-I$ since all these hyperedges contain $r$. By Lemma \ref{lemma:Btrees_struts} (\ref{part_1_lemma_Btrees_struts}), $V(\lset{J}_j)\subseteq F_{\leq r}$ for all $j$, and then $C\subseteq F_{\leq r}$.  Furthermore, $C\subseteq (F^{-_r\lset{I}})_{\leq r}$ since $V(\lset{J}_j)\cap F_{\leq \ell(u_i)}=\emptyset$ for every $i$---otherwise $u_i$ would not satisfy (\ref{equation:C_condition}). Thus $C$ connects $(F^{-_r\lset{I}})_{\leq r}$ in $H'{-\lset{I}}$, and $(F^{-_r\lset{I}})_{\leq r}$ is connected by $H'{-\lset{I}}$ as needed. Finally, that $(F^{-_r\lset{I}})_{\leq r}$ is maximally connected by $H'{-\lset{I}}$ follows since $(F^{-_r\lset{I}})_{\max}$ in (\ref{equation:Fmax_minus_I}) is a partition of $V$ and all $\{F_{\leq r_k}\}_{k = 1}^{M}$ and $\{F_{\leq \ell(u_i)}\}_{i=1}^m$ are maximal in $H'{-\lset{I}}$, except perhaps $(F^{-_r\lset{I}})_{\leq r}$, then $(F_{\max})^{-_r\lset{I}} = (H'{-\lset{I}})_{\max}$.

Since the conditions of Lemma \ref{lemma:H_Forest_removing_root} hold in this case, we find that $F^{-_r\lset{I}}$ is an $(H{-\lset{I}})$-forest.

    \item \textbf{If $\lset{I}\not\in H'$}:
We have that \begin{align}
\nonumber (F^{+_r\lset{I}})_{\max} =& (\{F_{\leq r_1}, \dots , F_{\leq r_M}\}\setminus\{F_{\leq \rho(u)} \mid u \in V(\lset{I})\})\,\cup  \{(F^{+_r\lset{I}})_{\leq r}\}\\
\label{equation:F_mas_plus_I}=& \{F_{\leq r_{t_1}}, \dots , F_{\leq r_{t_k}}\}\,\cup  \{(F^{+_r\lset{I}})_{\leq r}\},
\end{align} where $$(F^{+_r\lset{I}})_{\leq r} = \bigcup_{u \in I} F_{\leq \rho(u)}\in H'{+\lset{I}},$$ 
and $r_{t_1},\dots, r_{t_k}$ are the roots of trees in $F$ that do not intersect $V(\lset{I})$.

The construction of $F^{+_r\lset{I}}$ does not affect any of the $F_{\leq r_{t_j}}$, and they remain maximal in $H'{+\lset{I}}$.
We also have that, for all $u\in V(\lset{I})$, $F_{\leq \rho(u)}$ is connected by $H' \subset H'{+\lset{I}}$ and $ V(\lset{I})\cap F_{\leq \rho(u)}\neq \emptyset$. Thus $V(\lset{I})$ connects $\{F_{\leq \rho(u)}\}_{u\in V(\lset{I})}$ in $H'{+\lset{I}}$ and $(F^{+_r\lset{I}})_{\leq r}$ is connected by $H'+\lset{I}$.  Moreover, $(F^{+_r\lset{I}})_{\leq r}$ has to be also maximal in $H'{+\lset{I}}$ since the blocks in (\ref{equation:F_mas_plus_I}) form a partition of $V$. 

We conclude that $(H'{+\lset{I}})_{\max}=(F^{+_r\lset{I}})_{\max}$ and, by Lemma \ref{lemma:H_Forest_removing_root}, $F^{+_r\lset{I}}$ is an $(H'{+\lset{I}})$-forest, as required. 
\end{enumerate}

This concludes the proof of (1) and (2). To prove (3), note  that (1) and (2) already imply that $(F^{-_r\lset{I}})^{+_r\lset{I}}$ and $(F^{+I})^{-_r\lset{I}}$ are $H'$-forests and so
\begin{equation}\label{equation:pm_forests_max}
    ((F^{-_r\lset{I}})^{+_r \lset{I}})_{\max}=((F^{+_r\lset{I}})^{-_r\lset{I}})_{\max}=F_{\max}=H'_{\max}.
\end{equation}
By the constructions of Definition \ref{definition:FpmI_forests} we have that the edges satisfy $$E((F^{-_r\lset{I}})^{+_r \lset{I}})=E(F^{-_r\lset{I}}) \cup E_2 = \left(E(F) \setminus E_1 \right) \cup E_2,$$ where \begin{align*}
E_1&=\{ \{r,\ell_F(w)\} \mid w\in V(\lset{I}) \setminus \{r\} \text{ satisfies (\ref{equation:C_condition})} \},\\    
E_2&=\{\{r,\rho_{F^{-_r\lset{I}}}(w)\} \mid w \in V(\lset{I})\text{ and }  \rho_{F^{-_r\lset{I}}}(w)\not = r\}.
\end{align*}

 From (\ref{equation:different_vertices_of_I}) we have that $\{w\in V(\lset{I}) \mid \rho_{F^{-_r\lset{I}}}(w)\neq r\} = \{u_1,\dots,u_m\}$, which are precisely the elements of $V(\lset{I})\setminus r$ satisfying (\ref{equation:C_condition}). Since $\ell_F(u_i)$ is a root of $F^{-_r\lset{I}}$, $\rho_{F^{-_r\lset{I}}}(u_i)=\ell_F(u_i)$. Thus $E_1=E_2$, $E((F^{-_r\lset{I}})^{+_r \lset{I}})=E(F)$, and $(F^{-_r\lset{I}})^{+_r \lset{I}}=F$, as claimed. 

On the other hand, \begin{align*}
E((F^{+_r \lset{I}})^{-_r\lset{I}}) =& E(F^{+_r \lset{I}}) \setminus E_2'  =  \left(E(F) \cup E_1'\right)\setminus E_2',
    \end{align*} where \begin{align*}
        E_1'&=\{\{ r ,\rho_F(w)\} \mid w \in V(\lset{I})\text{ and }  \rho_F(w)\not = r \},\\ E_2'&=\{\{ r,\ell_{F^{+_r \lset{I}}}(w)\}\mid w \in V(\lset{I}) \text{ satisfies  \hyperref[equation:C_condition]{($C_{H'{+\lset{I}},\lset{I},r}$)}}\}.
    \end{align*}

If $w\in V(\lset{I})$ and $\rho_F(w)\neq r$, then $\ell_{F^{+_r \lset{I}}}(w)=\rho_F(w)$ and $F^{+_r \lset{I}}_{\leq \ell_{F^{+_r \lset{I}}}(w)}=F_{\leq \rho_F(w)}$. In particular, the only $\lset{J}\in E_r(H'{+\lset{I}})$ satisfying $V(\lset{J})\cap F^{+_r \lset{I}}_{\leq \ell_{F^{+_r \lset{I}}}(w)}\neq \emptyset$ is $\lset{J}=\lset{I}$, since $F_{\leq \rho_F(w)}\in F_{\max}$. This means that $w$ satisfies \hyperref[equation:C_condition]{($C_{H'{+\lset{I}},\lset{I},r}$)} and thus $E_1'\subseteq E_2'$. 

Conversely, if $w\in V(\lset{I})$ satisfies  \hyperref[equation:C_condition]{($C_{H'{+\lset{I}},\lset{I},r}$)}, we have that   \eqref{equation:pm_forests_max} implies that $F^{+_r \lset{I}}_{\leq \ell_{F^{+_r \lset{I}}}(w)} \in ((F^{+_r \lset{I}})^{-_r\lset{I}})_{\max}=F_{\max}$, and thus $F^{+_r \lset{I}}_{\leq \ell_{F^{+_r \lset{I}}}(w)}=F_{\leq \rho_F(w)}$. Therefore, $E_2'\subseteq E_1'$, $E((F^{+_r \lset{I}})^{-_r \lset{I}})=E(F)$, and so $(F^{+_r \lset{I}})^{-_r \lset{I}}=F$  as we wanted to show.
\end{proof}

\begin{lemma}\label{lemma:fixed_point_free_involution}
The map $\psi$ in \eqref{equation:sign_reversing_involution_definition} is a  sign-reversing involution on $\Omega_H$.
\end{lemma}

\begin{proof} Lemma \ref{lemma: F_pm_I_well_defined} (1) and (2) show that $\psi(H',F,r)\in \Omega_H$, for all $(H',F,r)\in \Omega_H$, so $\psi$ is well-defined. The fact that $\psi\circ\psi=\text{id}_{\Omega_H}$ is a direct consequence of Lemma \ref{lemma: F_pm_I_well_defined}(3). Let us write $\psi(H',F,r)=(H'',F',r)$.

To prove that $\psi$ is sign reversing, we need to show that $$ \omega(H',F,r)=(-1)^{|H'|}t^{\des (F)+\sigma_F(r)}=-(-1)^{|H''|}t^{\des (F')+\sigma_{F'}(r)}=-\omega(H'',F',r).$$ By construction $|H'| = |H''| \pm 1$, so $(-1)^{|H'|}=-(-1)^{|H''|}$. Thus, we are left to prove that \begin{equation}\label{equation:descents_plus_sigma}
\des (F) + \sigma_F (r) = \des (F')+ \sigma_{F'} (r).\end{equation}  Let $r_0<\cdots<r_{\kappa(H')-1}$ be the roots of $F$ ordered according to the order in $V$.

We first consider the case $\lset{I}_r\in H'$. Let us write $V(\lset{I}_r)=\{ r,u_1,\dots,u_m,v_1,\dots,v_p\}$ as in (\ref{equation:different_vertices_of_I}), where the $u_i$ satisfy (\ref{equation:C_condition}) and the remaining $v_j$ do not. Recall that $F^{-_r\lset{I}_r}$ has roots $r_0,\dots,r_{\kappa(H')-1}, \ell_F(u_1),\dots,\ell_F(u_m)$. By construction, $\des (F^{-_r\lset{I}_r})=\des (F)-N$, where $N=\#\{i \mid \ell_F(u_i)< r \}$. Moreover, $\sigma_{F^{-_r\lset{I}_r}}(r)=\sigma_F(r)+N$, since $r$ is moved $N$ positions to the right in the  order of the roots of $F^{-_r\lset{I}_r}$. Therefore, $$\des (F) + \sigma_F (r)=(\des (F^{-_r\lset{I}_r})+N)+\sigma_F(r)=\des (F^{-_r\lset{I}_r})+\sigma_{F^{-_r\lset{I}_r}}(r).$$ 

Now assume that $\lset{I}_r\not\in H'$. Note that $F^{+_r \lset{I}_r}$ has roots $\{r_0,\dots, r_{\kappa(H')-1}\}\setminus\{\rho_F(w) \mid w\in V(\lset{I}_r), w\neq r \}$. In this case $\des (F^{+_r \lset{I}_r})=\des (F)+N$ and $\sigma_{F^{+_r \lset{I}_r}}(r)=\sigma_F(r)-N$, where now $N=\#\{w\in V(\lset{I}_r) \mid \rho_F(w)<r\}$. Thus, (\ref{equation:descents_plus_sigma}) holds in this case as well. The proof is complete.
\end{proof}

\section{Proofs of Theorems \ref{theorem:h_polynomial_as_mobius_values} and \ref{Theorem:second_recurrence_h}}\label{section:proofs_of_main_results}

These results will follow from the next analogue of Theorem \ref{Theorem:second_recurrence_h}, where $h$-polynomials are indexed by hypergraphs and computed recursively.

\begin{proposition}	\label{proposition:second_recurrence_h_hypergraphs} The $h$-polynomials associated to hypergraphs are determined by the following recurrence relations:
	\begin{enumerate}
		\item If $H=\emptyset$, $h_{H}(t)=1$. 

        \item If $H$ has connected components $H_1,\dots,H_k$, then
		$$h_H(t)=h_{H_1}(t)\cdots h_{H_k}(t).$$

        \item If $H$ is a connected hypergraph with $|V(H)|> 1$, then $h_{H}(t)$ satisfies the two recurrences:
    \begin{align}\label{equation:hrecursion1_hypergraphs}
    h_H(t) = \sum_{H'< H} (-1)^{|H|-|H'|-1}[\kappa(H')]_t h_{H'}(t),
\end{align}
and
\begin{align}\label{equation:hrecursion2_hypergraphs}
    h_{H}(t) = \sum_{\emptyset\neq H' \le H}(-1)^ {|H'|-1} h_{H/_{H'}}(t)\prod_{I\in H'_{\max}} [\,|I|\,]_t.
\end{align}
	\end{enumerate}
\end{proposition}

\begin{proof} Formulas (1) and (2) are valid thanks to Proposition \ref{proposition:recurrence_postnikov}, and   \eqref{equation:hrecursion1_hypergraphs} follows from Proposition \ref{proposition:main_argument} by solving for $h_H(t)$. Therefore, the family $\{h_H\}$ of $h$-polynomials of hypergraphs is completely determined by (1), (2) and \eqref{equation:hrecursion1_hypergraphs}. But  Propositions \ref{proposition:recurrenceM} shows that the same recurrences are satisfied by the family of  auxiliary polynomials $\{M_H\}$. By uniqueness, we see that \begin{equation}\label{equation:h_equals_M_up_to_sign}
(-1)^{|H|} h_H(t) = M_H(t),\qquad \text{for all hypergraphs } H.\end{equation}  Finally, \eqref{equation:hrecursion2_hypergraphs} follows from \eqref{equation:M_lemma_2} by using \eqref{equation:h_equals_M_up_to_sign} and solving for $h_H$ which corresponds to $H'=\emptyset$. Note that here we used that $|H/_{H'}|=|H|-|H'|$ and that $H/_{\emptyset}$ is isomorphic to $H$.
\end{proof}

The previous proof includes the proof of Theorem \ref{theorem:h_polynomial_as_mobius_values}, which is exactly \eqref{equation:h_equals_M_up_to_sign}. Moreover, Theorem \ref{Theorem:second_recurrence_h} is a particular case of Proposition \ref{proposition:second_recurrence_h_hypergraphs}. Indeed, if $\B$ is a connected building set, $E(\B)$ is a connected hypergraph. Now, if $H'\leq E(\B)$, then $H'$ is of the form $H'=E(\A)$, for a unique spanning sub-building set $\A\leq \B$. Moreover, $H'$ and $\A$ have the same number of connected components. Therefore, \eqref{equation:hrecursion1_hypergraphs} gives \eqref{equation:hrecursion1} when applied to $H=E(\B)$ since $h_H(t)=h_{E(\B)}(t)=h_\B(t)$. In the same way, \eqref{equation:hrecursion2_hypergraphs} gives  \eqref{equation:hrecursion2} since by Definition \ref{definition:building_set_contraction} the struts of the building set $\B/_\A$ coincide with $E(E(\B)/_{E(\A)})$ for the contracted hypergraph $E(\B)/_{E(\A)}$. Thus $h_{\B/_\A}(t)=h_{E(\B/_\A)}(t)=h_{E(E(\B)/_{E(\A)})}(t)=h_{E(\B)/_{E(\A)}}(t)$, giving the required recursion.

\begin{corollary} For any two hypergraphs $H,H'$ on $V$ with  $\B(H)=\B(H')$, we have that
$$(-1)^{|H|} M_{H}(t) = (-1)^{|H'|} M_{H'}(t).$$
In particular, 
$$h_H(t)=(-1)^{|H|} M_{H}(t) = (-1)^{|E(H)|} M_{E(H)}(t)= (-1)^{|\B(H)|} M_{\B(H)}(t).$$ 
\end{corollary}

\section{Examples}\label{section:examples}
We illustrate the use of Theorem \ref{Theorem:second_recurrence_h} with two examples. 
The strength of the new recursion can be appreciated in the cases where $|E(\B)|<|V(\B)|$. 

\begin{example} Consider the building set 
\begin{equation*}
    \B_n:=\{\{i\}\mid i \in [n]\}\cup \{[n]\},
\end{equation*} with ground set $[n]$ and only strut $[n]$, so $P_{\B_n}$ is an $(n-1)$-simplex. If $\A<\B_n$, then $E(\A)=\emptyset$ and $|\kappa(A)|=n$. Therefore, we obtain directly from recurrence \eqref{equation:hrecursion1} that $$h_{\B_n}(t)=(-1)^{1-0-1}[n]_t=[n]_t,$$ which gives a short proof of this well-known result, see \cite[Example 6.11]{PostnikovReinerWilliams2007}.
\end{example}

\begin{example} Fix integers $0\leq k\leq n$ and consider the building set $\mathcal{T}_{n,k}$ with ground set $[n+k]$ and struts $\{l,l+n\}, l=1,\dots,k$ and $[n]$, see Figure \ref{figure:turtles_buildingsets} for some examples.  Note that if $k=0$ we simply have $\mathcal{T}_{n,0}=\B_n$.

\begin{figure}[ht]
    \centering
    \captionsetup{justification=centering}
    \resizebox{1\linewidth}{!}{\definecolor{color_0}{HTML}{FFFFFF}
\definecolor{color_1}{HTML}{2A2A2A}
\definecolor{color_2}{HTML}{000000}

\begin{tikzpicture}[ main_node_g/.style={circle,minimum size=0.5em,inner sep=2pt, draw=black, fill = white}]

\node at (-0.25,2.25) {$\mathcal{T}_{4,2}$};
\node at (4,2.25) {$\mathcal{T}_{7,7}$};
\node at (8,2.25) {$\mathcal{T}_{5,4}$};

\node[main_node_g, fill = color_0,text = color_1,minimum size=0.535em] (0) at (0.60,-0.30) {\tiny \textbf{1}};
\node[main_node_g, fill = color_0,text = color_1,minimum size=0.535em] (1) at (-0.30,0.60) {\tiny \textbf{2}};
\node[main_node_g, fill = color_0,text = color_1,minimum size=0.535em] (2) at (-1.20,-0.30) {\tiny \textbf{3}};
\node[main_node_g, fill = color_0,text = color_1,minimum size=0.535em] (3) at (-0.30,-1.20) {\tiny \textbf{4}};
\node[main_node_g, fill = color_0,text = color_1,minimum size=0.535em] (4) at (1.50,-0.30) {\tiny \textbf{5}};
\node[main_node_g, fill = color_0,text = color_1,minimum size=0.535em] (5) at (-0.30,1.50) {\tiny \textbf{6}};

\begin{pgfonlayer}{background}
\fill[fill = color_2, opacity = 0.3, line width = 0.1mm, draw = color_2, solid]
(-0.47,0.78)--(-1.38,-0.13)--(-1.39,-0.14)--(-1.40,-0.16)--(-1.41,-0.17)--(-1.41,-0.18)--(-1.42,-0.20)--(-1.43,-0.21)--(-1.43,-0.22)--(-1.44,-0.24)--(-1.44,-0.25)--(-1.44,-0.27)--(-1.44,-0.28)--(-1.44,-0.30)--(-1.44,-0.32)--(-1.44,-0.33)--(-1.44,-0.35)--(-1.44,-0.36)--(-1.43,-0.38)--(-1.43,-0.39)--(-1.42,-0.40)--(-1.41,-0.42)--(-1.41,-0.43)--(-1.40,-0.44)--(-1.39,-0.46)--(-1.38,-0.47)--(-0.47,-1.38)--(-0.46,-1.39)--(-0.44,-1.40)--(-0.43,-1.41)--(-0.42,-1.41)--(-0.40,-1.42)--(-0.39,-1.43)--(-0.38,-1.43)--(-0.36,-1.44)--(-0.35,-1.44)--(-0.33,-1.44)--(-0.32,-1.44)--(-0.30,-1.44)--(-0.28,-1.44)--(-0.27,-1.44)--(-0.25,-1.44)--(-0.24,-1.44)--(-0.22,-1.43)--(-0.21,-1.43)--(-0.20,-1.42)--(-0.18,-1.41)--(-0.17,-1.41)--(-0.16,-1.40)--(-0.14,-1.39)--(-0.13,-1.38)--(0.78,-0.47)--(0.79,-0.46)--(0.80,-0.44)--(0.81,-0.43)--(0.81,-0.42)--(0.82,-0.40)--(0.83,-0.39)--(0.83,-0.38)--(0.84,-0.36)--(0.84,-0.35)--(0.84,-0.33)--(0.84,-0.32)--(0.84,-0.30)--(0.84,-0.28)--(0.84,-0.27)--(0.84,-0.25)--(0.84,-0.24)--(0.83,-0.22)--(0.83,-0.21)--(0.82,-0.20)--(0.81,-0.18)--(0.81,-0.17)--(0.80,-0.16)--(0.79,-0.14)--(0.78,-0.13)--(-0.13,0.78)--(-0.14,0.79)--(-0.16,0.80)--(-0.17,0.81)--(-0.18,0.81)--(-0.20,0.82)--(-0.21,0.83)--(-0.22,0.83)--(-0.24,0.84)--(-0.25,0.84)--(-0.27,0.84)--(-0.28,0.84)--(-0.30,0.84)--(-0.32,0.84)--(-0.33,0.84)--(-0.35,0.84)--(-0.36,0.84)--(-0.38,0.83)--(-0.39,0.83)--(-0.40,0.82)--(-0.42,0.81)--(-0.43,0.81)--(-0.44,0.80)--(-0.46,0.79)--cycle;
\end{pgfonlayer}

\path[draw  = color_2 ,thick, solid](0) edge node {} (4);
\path[draw  = color_2 ,thick, solid](1) edge node {} (5);

\node[main_node_g, fill = color_0,text = color_1,minimum size=0.535em] (0) at (4.75,0.00) {\tiny \textbf{1}};
\node[main_node_g, fill = color_0,text = color_1,minimum size=0.535em] (1) at (4.47,0.59) {\tiny \textbf{2}};
\node[main_node_g, fill = color_0,text = color_1,minimum size=0.535em] (2) at (3.83,0.73) {\tiny \textbf{3}};
\node[main_node_g, fill = color_0,text = color_1,minimum size=0.535em] (3) at (3.32,0.33) {\tiny \textbf{4}};
\node[main_node_g, fill = color_0,text = color_1,minimum size=0.535em] (4) at (3.32,-0.33) {\tiny \textbf{5}};
\node[main_node_g, fill = color_0,text = color_1,minimum size=0.535em] (5) at (3.83,-0.73) {\tiny \textbf{6}};
\node[main_node_g, fill = color_0,text = color_1,minimum size=0.535em] (6) at (4.47,-0.59) {\tiny \textbf{7}};
\node[main_node_g, fill = color_0,text = color_1,minimum size=0.535em] (7) at (5.50,0.00) {\tiny \textbf{8}};
\node[main_node_g, fill = color_0,text = color_1,minimum size=0.535em] (8) at (4.94,1.17) {\tiny \textbf{9}};
\node[main_node_g, fill = color_0,text = color_1,minimum size=0.535em] (9) at (3.67,1.46) {\tiny \textbf{10}};
\node[main_node_g, fill = color_0,text = color_1,minimum size=0.535em] (10) at (2.65,0.65) {\tiny \textbf{11}};
\node[main_node_g, fill = color_0,text = color_1,minimum size=0.535em] (11) at (2.65,-0.65) {\tiny \textbf{12}};
\node[main_node_g, fill = color_0,text = color_1,minimum size=0.535em] (12) at (3.67,-1.46) {\tiny \textbf{13}};
\node[main_node_g, fill = color_0,text = color_1,minimum size=0.535em] (13) at (4.94,-1.17) {\tiny \textbf{14}};

\begin{pgfonlayer}{background}
\fill[fill = color_2, opacity = 0.3, line width = 0.1mm, draw = color_2, solid]
(3.08,0.33)--(3.08,-0.33)--(3.08,-0.34)--(3.08,-0.36)--(3.08,-0.37)--(3.09,-0.39)--(3.09,-0.40)--(3.10,-0.42)--(3.10,-0.43)--(3.11,-0.44)--(3.12,-0.46)--(3.13,-0.47)--(3.14,-0.48)--(3.15,-0.49)--(3.16,-0.50)--(3.17,-0.51)--(3.68,-0.92)--(3.69,-0.93)--(3.70,-0.94)--(3.72,-0.95)--(3.73,-0.95)--(3.74,-0.96)--(3.76,-0.96)--(3.77,-0.97)--(3.79,-0.97)--(3.80,-0.97)--(3.82,-0.98)--(3.83,-0.98)--(3.85,-0.98)--(3.86,-0.97)--(3.88,-0.97)--(3.89,-0.97)--(4.53,-0.82)--(4.54,-0.82)--(4.56,-0.81)--(4.57,-0.81)--(4.59,-0.80)--(4.60,-0.79)--(4.61,-0.78)--(4.62,-0.77)--(4.64,-0.76)--(4.65,-0.75)--(4.66,-0.74)--(4.67,-0.73)--(4.67,-0.72)--(4.68,-0.70)--(4.69,-0.69)--(4.97,-0.10)--(4.98,-0.09)--(4.98,-0.08)--(4.99,-0.06)--(4.99,-0.05)--(4.99,-0.03)--(4.99,-0.02)--(4.99,0.00)--(4.99,0.02)--(4.99,0.03)--(4.99,0.05)--(4.99,0.06)--(4.98,0.08)--(4.98,0.09)--(4.97,0.10)--(4.69,0.69)--(4.68,0.70)--(4.67,0.72)--(4.67,0.73)--(4.66,0.74)--(4.65,0.75)--(4.64,0.76)--(4.62,0.77)--(4.61,0.78)--(4.60,0.79)--(4.59,0.80)--(4.57,0.81)--(4.56,0.81)--(4.54,0.82)--(4.53,0.82)--(3.89,0.97)--(3.88,0.97)--(3.86,0.97)--(3.85,0.98)--(3.83,0.98)--(3.82,0.98)--(3.80,0.97)--(3.79,0.97)--(3.77,0.97)--(3.76,0.96)--(3.74,0.96)--(3.73,0.95)--(3.72,0.95)--(3.70,0.94)--(3.69,0.93)--(3.68,0.92)--(3.17,0.51)--(3.16,0.50)--(3.15,0.49)--(3.14,0.48)--(3.13,0.47)--(3.12,0.46)--(3.11,0.44)--(3.10,0.43)--(3.10,0.42)--(3.09,0.40)--(3.09,0.39)--(3.08,0.37)--(3.08,0.36)--(3.08,0.34)--cycle;
\end{pgfonlayer}

\path[draw  = color_2 ,thick, solid](13) edge node {} (6);

\path[draw  = color_2 ,thick, solid](12) edge node {} (5);

\path[draw  = color_2 ,thick, solid](7) edge node {} (0);

\path[draw  = color_2 ,thick, solid](10) edge node {} (3);

\path[draw  = color_2 ,thick, solid](11) edge node {} (4);

\path[draw  = color_2 ,thick, solid](8) edge node {} (1);

\path[draw  = color_2 ,thick, solid](9) edge node {} (2);

\node[main_node_g, fill = color_0,text = color_1,minimum size=0.535em] (0) at (8.77,-0.15) {\tiny \textbf{1}};
\node[main_node_g, fill = color_0,text = color_1,minimum size=0.535em] (1) at (8.27,0.54) {\tiny \textbf{2}};
\node[main_node_g, fill = color_0,text = color_1,minimum size=0.535em] (2) at (7.46,0.27) {\tiny \textbf{3}};
\node[main_node_g, fill = color_0,text = color_1,minimum size=0.535em] (3) at (7.46,-0.58) {\tiny \textbf{4}};
\node[main_node_g, fill = color_0,text = color_1,minimum size=0.535em] (4) at (8.27,-0.84) {\tiny \textbf{5}};
\node[main_node_g, fill = color_0,text = color_1,minimum size=0.535em] (5) at (9.50,-0.15) {\tiny \textbf{6}};
\node[main_node_g, fill = color_0,text = color_1,minimum size=0.535em] (6) at (8.50,1.23) {\tiny \textbf{7}};
\node[main_node_g, fill = color_0,text = color_1,minimum size=0.535em] (7) at (6.88,0.70) {\tiny \textbf{8}};
\node[main_node_g, fill = color_0,text = color_1,minimum size=0.535em] (8) at (6.88,-1.01) {\tiny \textbf{9}};

\begin{pgfonlayer}{background}
\fill[fill = color_2, opacity = 0.3, line width = 0.1mm, draw = color_2, solid]
(7.22,0.27)--(7.22,-0.58)--(7.22,-0.59)--(7.22,-0.61)--(7.22,-0.63)--(7.23,-0.64)--(7.23,-0.66)--(7.24,-0.67)--(7.24,-0.68)--(7.25,-0.70)--(7.26,-0.71)--(7.27,-0.72)--(7.27,-0.74)--(7.28,-0.75)--(7.30,-0.76)--(7.31,-0.77)--(7.32,-0.78)--(7.33,-0.79)--(7.35,-0.79)--(7.36,-0.80)--(7.37,-0.81)--(7.39,-0.81)--(8.20,-1.08)--(8.21,-1.08)--(8.23,-1.08)--(8.24,-1.09)--(8.26,-1.09)--(8.27,-1.09)--(8.29,-1.09)--(8.30,-1.09)--(8.32,-1.08)--(8.33,-1.08)--(8.35,-1.08)--(8.36,-1.07)--(8.38,-1.06)--(8.39,-1.06)--(8.41,-1.05)--(8.42,-1.04)--(8.43,-1.03)--(8.44,-1.02)--(8.45,-1.01)--(8.46,-1.00)--(8.47,-0.99)--(8.97,-0.30)--(8.98,-0.28)--(8.99,-0.27)--(9.00,-0.26)--(9.00,-0.24)--(9.01,-0.23)--(9.01,-0.21)--(9.02,-0.20)--(9.02,-0.18)--(9.02,-0.17)--(9.02,-0.15)--(9.02,-0.14)--(9.02,-0.12)--(9.02,-0.11)--(9.01,-0.09)--(9.01,-0.08)--(9.00,-0.06)--(9.00,-0.05)--(8.99,-0.04)--(8.98,-0.02)--(8.97,-0.01)--(8.47,0.68)--(8.46,0.69)--(8.45,0.70)--(8.44,0.71)--(8.43,0.72)--(8.42,0.73)--(8.41,0.74)--(8.39,0.75)--(8.38,0.76)--(8.36,0.76)--(8.35,0.77)--(8.33,0.77)--(8.32,0.78)--(8.30,0.78)--(8.29,0.78)--(8.27,0.78)--(8.26,0.78)--(8.24,0.78)--(8.23,0.78)--(8.21,0.77)--(8.20,0.77)--(7.39,0.51)--(7.37,0.50)--(7.36,0.49)--(7.35,0.49)--(7.33,0.48)--(7.32,0.47)--(7.31,0.46)--(7.30,0.45)--(7.28,0.44)--(7.27,0.43)--(7.27,0.42)--(7.26,0.40)--(7.25,0.39)--(7.24,0.38)--(7.24,0.36)--(7.23,0.35)--(7.23,0.33)--(7.22,0.32)--(7.22,0.30)--(7.22,0.29)--cycle;
\end{pgfonlayer}

\path[draw  = color_2 ,thick, solid](7) edge node {} (2);

\path[draw  = color_2 ,thick, solid](8) edge node {} (3);

\path[draw  = color_2 ,thick, solid](1) edge node {} (6);

\path[draw  = color_2 ,thick, solid](0) edge node {} (5);

\end{tikzpicture}}
	\caption{The building sets  $\mathcal{T}_{4,2}$, $\mathcal{T}_{7,7}$, and  $\mathcal{T}_{5,4}$.}
    \label{figure:turtles_buildingsets}
\end{figure}
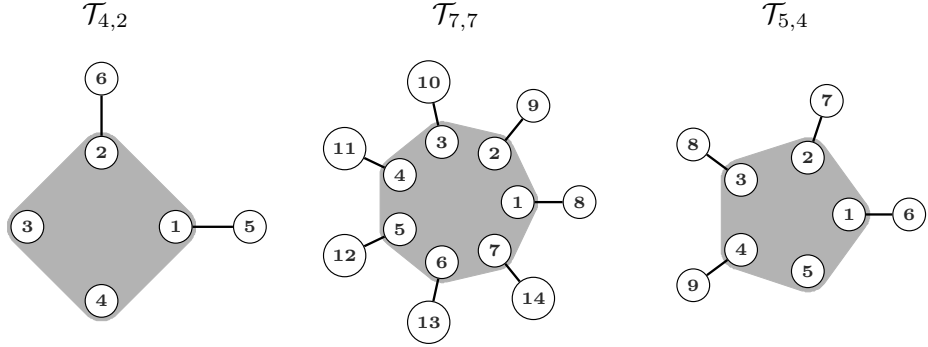

In order to compute $h_{\T_{n,k}}$ from  \eqref{equation:hrecursion1} we describe all $A\leq \T_{n,k}$. There are two possibilities. First, $[n]\in E(\A)$. In this case, having $\A$ means choosing $0\leq j \leq k$ lines to remove, thus $|E(\A)|= k - j + 1$,  $\kappa(\A) = j + 1$, and $h_\A=h_{\T_{n,k-j}}$. Second, if $[n]\not\in E(\A)$, removing $0\leq j \leq k$ lines we have that $|E(\A)|= k - j$,  $\kappa(\A) = (k - j) + 2j + (n - k) = n+j$, and $h_\A=h_{P_2}^{k-j}=(1+t)^{k-j}$. In conclusion, \eqref{equation:hrecursion1} leads to the equation $$\sum_{j=0}^k \binom{k}{j} (-1)^{k-j+1} {[j+1]_t} h_{\mathcal{T}_{n,k-j}}(t) + \sum_{j=0}^k \binom{k}{j} (-1)^{k-j}  {[n+j]_t} (1+t)^{k-j} = 0.$$ Let us write this as 
$$\sum_{j=0}^k \binom{k}{j}  {[j+1]_t} (-1)^{k-j}h_{\T_{n,k-j}}(t) = \sum_{j=0}^k \binom{k}{j} {[n+j]_t} {(-1)^{k-j}(1+t)^{k-j}},\, 0\leq k\leq n.$$ Setting $T(t,z)=\sum_{j=0}^n (-1)^j h_{\T_{n,j}}(t)\frac{z^j}{j!}$ and using the exponential generating series $$\sum_{j=0}^\infty \frac{[j+n]_t}{j!}z^j=\frac{e^z-t^ne^{tz}}{1-t},\qquad  \sum_{j=0}^\infty \frac{(-1)^j(1+t)^j}{j!}z^j=e^{-(1+t)z},$$ the previous equation translates into $$\frac{e^z-te^{tz}}{1-t}T(t,z) = \frac{e^z-t^ne^{tz}}{1-t}e^{-(1+t)z},\quad (\text{mod } z^{n+1}).$$

This means that $$T(t,z) = \frac{e^z-t^ne^{tz}}{e^z-te^{tz}}e^{-(1+t)z} = \frac{1-t^ne^{(t-1)z}}{1-te^{(t-1)z}}e^{-(1+t)z} \quad (\text{mod } z^{n+1}).$$ By writing \begin{align*}
\frac{t^ne^{(t-1)z}-1}{te^{(t-1)z}-1} = t^{n-1} + \frac{t^{n-1}-1}{te^{(t-1)z}-1} = t^{n-1} + [n-1]_{t}\frac{t-1}{t-e^{-(t-1)z}}e^{-(t-1)z},
\end{align*} we find that $$T(t,z) = t^{n-1}e^{-(1+t)z} + [n-1]_{t}\frac{t-1}{t-e^{-(t-1)z}}e^{-2tz} \quad (\text{mod } z^{n+1}).$$ Recalling the exponential generating series of the Eulerian polynomials, see e.g. \cite[Theorem 1.6]{Petersen2015}, $$\sum_{m=0}^\infty A_m(t)\frac{z^m}{m!}=\frac{t-1}{t-e^{(t-1)z}},$$ we finally find that $$ h_{\T_{n,j}}(t) =
(1+t)^jt^{n-1}+[n-1]_t\sum_{l=0}^j \binom{j}{l} (2t)^{j-l} A_{l}(t).$$

For instance, we have the values  $$h_{\T_{n,0}}(t)=[n]_t,\qquad h_{\T_{n,1}}(t)=(2t+1)[n]_t-t^n = 1+3(t+\cdots+t^{n-1})+t^n,$$ where the last equality holds for $n\geq 2$.

\end{example}

\section*{Acknowledgements}
The authors would like to thank Luis Alfonso Contreras Ordo\~nez, F\'elix G\'elinas, Mar\'ia Ronco, Yannic Vargas, and Michelle L. Wachs for helpful conversations on different aspects of  nestohedra that clarified the exposition of these results.

\printbibliography

\end{document}